\documentclass[10pt]{amsart}
\usepackage{latexsym}
\usepackage{amssymb}
\usepackage{amscd}
\usepackage{epsf}
\usepackage{amsmath,amssymb,graphicx,amsthm,mathrsfs,verbatim}
\usepackage{pstricks}

\begin{document}

\theoremstyle{plain}
\newtheorem{Thm}{Theorem}[section]
\newtheorem{TitleThm}[Thm]{}
\newtheorem{Corollary}[Thm]{Corollary}
\newtheorem{Proposition}[Thm]{Proposition}
\newtheorem{Lemma}[Thm]{Lemma}
\newtheorem{Conjecture}[Thm]{Conjecture}
\theoremstyle{definition}
\newtheorem{Definition}[Thm]{Definition}
\theoremstyle{definition}
\newtheorem{Example}[Thm]{Example}
\newtheorem{TitleExample}[Thm]{}
\newtheorem{Remark}[Thm]{Remark}
\newtheorem{SimpRemark}{Remark}
\renewcommand{\theSimpRemark}{}

\numberwithin{equation}{section}

\newcommand{\C}{{\mathbb C}}
\newcommand{\Q}{{\mathbb Q}}
\newcommand{\R}{{\mathbb R}}
\newcommand{\Z}{{\mathbb Z}}
\newcommand{\mbS}{{\mathbb S}}
\newcommand{\mbU}{{\mathbb U}}
\newcommand{\mbO}{{\mathbb O}}
\newcommand{\mbG}{{\mathbb G}}
\newcommand{\mbH}{{\mathbb H}}

\newcommand{\flushpar}{\par \noindent}

\newcommand{\proj}{{\rm proj}}
\newcommand{\coker}{{\rm coker}\,}
\newcommand{\supp}{{\rm supp}\,}
\newcommand{\codim}{{\operatorname{codim}}}
\newcommand{\sing}{{\operatorname{sing}}}
\newcommand{\Tor}{{\operatorname{Tor}}}
\newcommand{\Hom}{{\operatorname{Hom}}}
\newcommand{\wt}{{\operatorname{wt}}}
\newcommand{\dlog}{{\operatorname{Derlog}}}
\newcommand{\Olog}[2]{\Omega^{#1}(\text{log}#2)}
\newcommand{\produnion}{\cup \negmedspace \negmedspace 
\negmedspace\negmedspace {\scriptstyle \times}}
\newcommand{\pd}[2]{\dfrac{\partial#1}{\partial#2}}

\def \ba {\mathbf {a}}
\def \bb {\mathbf {b}}
\def \bc {\mathbf {c}}
\def \bd {\mathbf {d}}
\def \bone {\boldsymbol {1}}
\def \bg {\mathbf {g}}
\def \bG {\mathbf {G}}
\def \bh {\mathbf {h}}
\def \bk {\mathbf {k}}
\def \bm {\mathbf {m}}
\def \bn {\mathbf {n}}
\def \bt {\mathbf {t}}
\def \bu {\mathbf {u}}
\def \bv {\mathbf {v}}
\def \bV {\mathbf {V}}
\def \bx {\mathbf {x}}
\def \bw {\mathbf {w}}
\def \b1 {\mathbf {1}}
\def \bga {\boldsymbol \alpha}
\def \bgb {\boldsymbol \beta}
\def \bgg {\boldsymbol \gamma}

\def \itc {\text{\it c}}
\def \ite {\text{\it e}}
\def \ith {\text{\it h}}
\def \iti {\text{\it i}}
\def \itj {\text{\it j}}
\def \itm {\text{\it m}}
\def \itM {\text{\it M}} 
\def \itn {\text{\it n}}
\def \ithn {\text{\it hn}}
\def \itt {\text{\it t}}

\def \cA {\mathcal{A}}
\def \cB {\mathcal{B}}
\def \cC {\mathcal{C}}
\def \cD {\mathcal{D}}
\def \cE {\mathcal{E}}
\def \cF {\mathcal{F}}
\def \cG {\mathcal{G}}
\def \cH {\mathcal{H}}
\def \cK {\mathcal{K}}
\def \cL {\mathcal{L}}
\def \cM {\mathcal{M}}
\def \cN {\mathcal{N}}
\def \cO {\mathcal{O}}
\def \cP {\mathcal{P}}
\def \cS {\mathcal{S}}
\def \cT {\mathcal{T}}
\def \cU {\mathcal{U}}
\def \cV {\mathcal{V}}
\def \cW {\mathcal{W}}
\def \cX {\mathcal{X}}
\def \cY {\mathcal{Y}}
\def \cZ {\mathcal{Z}}

\def \ga {\alpha}
\def \gb {\beta}
\def \gg {\gamma}
\def \gd {\delta}
\def \ge {\epsilon}
\def \gevar {\varepsilon}
\def \gk {\kappa}
\def \gl {\lambda}
\def \gs {\sigma}
\def \gt {\tau}
\def \gw {\omega}
\def \gz {\zeta}
\def \gG {\Gamma}
\def \gD {\Delta}
\def \gL {\Lambda}
\def \gS {\Sigma}
\def \gW {\Omega}

\def \dim {{\rm dim}\,}
\def \mod {{\rm mod}\;}
\def \rank {{\rm rank}\,}

\newcommand{\ds}{\displaystyle}
\newcommand{\vf}{\vspace{\fill}}
\newcommand{\vect}[1]{{\bf{#1}}}
\def\R{\mathbb R}
\def\C{\mathbb C}
\def\N{\mathbb N}
\def\Sym{\mathrm{Sym}}
\def\Sk{\mathrm{Sk}}
\def\GL{\mathrm{GL}}
\def\Diff{\mathrm{Diff}}
\def\id{\mathrm{id}}
\def\Pf{\mathrm{Pf}}
\def\gl{\mathfrak{gl}}
\def\sll{\mathfrak{sl}}
\def\g{\mathfrak{g}}
\def\h{\mathfrak{h}}
\def\k{\mathfrak{k}}
\def\t{\mathfrak{t}}
\def\OcN{\mathscr{O}_{\C^N}}
\def\Ocn{\mathscr{O}_{\C^n}}
\def\Ocm{\mathscr{O}_{\C^m}}
\def\Ocnz{\mathscr{O}_{\C^n,0}}
\def\Derlog{\mathrm{Derlog}\,}
\def\expdeg{\mathrm{exp\,deg}\,}

\title[Exceptional Orbit Hypersurfaces]
{Topology of Exceptional Orbit Hypersurfaces of Prehomogeneous Spaces}
\author[James Damon]{James Damon$^1$}

\thanks{(1) Partially supported by the Simons Foundation grant 230298 
and the National Science Foundation grant DMS-1105470}
\address{Department of Mathematics, University of North Carolina, Chapel 
Hill, NC 27599-3250, USA
}



\begin{abstract}
We consider the topology for a class of hypersurfaces with highly 
nonisolated singularites which arise as \lq\lq exceptional orbit 
varieties\rq\rq of a special class of prehomogeneous vector spaces, which 
are representations of linear algebraic groups with open orbits.  These 
hypersurface singularities include both determinantal hypersurfaces and 
linear free (and free*) divisors.  Although these hypersurfaces have highly 
nonisolated singularities, we determine the topology of their Milnor 
fibers, complements and links.  We do so by using the action of linear 
algebraic groups beginning with the complement, instead of using 
Morse-type arguments on the Milnor fibers.  This includes replacing the 
local Milnor fiber by a global Milnor fiber which has a \lq\lq complex 
geometry\rq\rq resulting from a transitive action of an appropriate 
algebraic group, yielding a compact \rq\rq model submanifold\rq\rq  for 
the homotopy type of the Milnor fiber.  The topology includes the 
(co)homology (in characteristic $0$, and $2$-torsion in one family) and 
homotopy groups, and we deduce the triviality of the monodromy 
transformations on rational (or complex) cohomology.  \par 
Unlike isolated singularities, the cohomology of the Milnor fibers and 
complements are isomorphic as algebras to exterior algebras or for one 
family, modules over exterior algebras; and cohomology of the link is, as a 
vector space, a truncated and shifted exterior algebra, for which the 
cohomology product structure is essentially trivial.  We also deduce from 
Bott\rq s periodicity theorem, the homotopy groups of the Milnor fibers 
for determinantal hypersurfaces in the \lq\lq stable range\rq\rq as the 
stable homotopy groups of the associated infinite dimensional symmetric 
spaces.  Lastly, we combine the preceding with a Theorem of Oka to obtain 
a class of \lq\lq formal linear combinations\rq\rq of exceptional orbit 
hypersurfaces which have Milnor fibers which are homotopy equivalent to 
joins of the compact model submanifolds.  It follows that Milnor fibers 
for all of these hypersurfaces are essentially never homotopy equivalent 
to bouquets of spheres (even allowing differing dimensions).  \par

\end{abstract} 
\maketitle

\section*{Introduction}  
\label{S:sec0} 
\par
In this paper we investigate the topology of a class of highly nonisolated 
hypersurface singularities $\cE,$ each of which arises as the 
hypersurface of exceptional orbits, the {\em exceptional orbit variety}, 
for a rational representation of connected complex linear algebraic group 
$\rho : G \to GL(V)$ with an open orbit $\cU$.  Such a space with group 
action has been studied by Sato and Kimura \cite{So}, \cite{SK} and is 
called by them a {\em prehomogeneous vector space}, which we will 
shorten in this paper to just {\em prehomogeneous space}.  Both {\em 
determinantal hypersurfaces} and {\em linear free divisors} belong to this 
class.  We consider how the topology of such singularities can be 
determined. \par
For nonisolated singularities with small dimensional singular set, a body 
of work by Siersma \cite{Si}, \cite{Si2}, \cite{Si3},  Tibar \cite{Ti}, 
Nemethi \cite{Ne}, Zaharia \cite{Z}, etc. has used Morse-theoretic methods 
to extend the results for isolated singularities, and show that the Milnor 
fiber is still homotopy equivalent to a bouquet of spheres, except now the 
spheres may have different dimensions.  One might ask to what extent 
such results apply to these hypersurfaces of exceptional orbits, which 
now have highly nonisolated singularities.  Complex Morse theory has been 
applied to determine the vanishing topology of \lq\lq nonlinear 
sections\rq\rq of these hypersurfaces in terms of \lq\lq singular Milnor 
fibers\rq\rq in results by Mond, Goryunov, Bruce, Pike, and this author, see 
e.g. \cite{DM}, \cite{Br}, \cite{D1, D2}, \cite{GM}, and \cite{DP3}.  
However, these results provide little
 information about the topology of the hypersurface singularities 
themselves. We shall see that the structure of these singularities are 
quite different from those studied in the above work.  In fact, to 
determine their topology we take a very different approach which makes 
substantial use of their representations via prehomogeneous spaces.  
\par   
For the exceptional orbit varieties, we will be concerned with the 
topology of the Milnor fiber, the complement and the link; and we 
determine their homotopy types along with the (co)homology structure, 
homotopy groups, and the monodromy action.  
The changes in approach which make this possible are: 
\begin{itemize}
\item[i)] reversing the usual approach from first determining the Milnor 
fiber and monodromy to then compute the topology of the link and 
complement by beginning instead with the complement and deducing the 
topology of the link and Milnor fiber.
\item[ii)] replacing the local Milnor fiber by a global Milnor fiber, which is 
a smooth affine hypersurface that has a \lq\lq model complex 
geometry\rq\rq resulting from the transitive action of an associated 
linear algebraic group, yielding as a deformation retract a compact 
submanifold; 
\item[iii)]  using the relation between the two algebraic group actions and 
the topology of maximal compact subgroups to deduce the cohomological 
triviality of an associated fibration of the groups; and 
\item[iv)] using the preceding to determine the topology and cohomology 
of the Milnor fiber.
\end{itemize}
\par 
We will explicity compute the cohomology of the Milnor fiber, the 
complement and the link for two classes: determinantal hypersurfaces, 
which are varieties of singular matrices in the spaces of $m \times m$ 
matrices which may be symmetric, general, or skew-symmetric (with $m$ 
even); and exceptional orbit varieties in the general equidimensional case.
  This second class will in particular apply to both linear free (and free*) 
divisors, introduced for reductive groups by Buchweitz-Mond \cite{BM} and 
determined for \lq\lq block representations\rq\rq of solvable groups in 
Damon-Pike \cite{DP}, \cite{DP2}.  \par 
We express the cohomology algebras of the Milnor fibers as either exterior 
algebras, or for the one case of symmetric $m \times m$ matrices with 
$m$ even, modules over an exterior algebra on two generators, see 
Theorems~\ref{Thm2.1}  and \ref{Thm4.2}.
For determinantal hypersurfaces, we specifically show, 
Theorem~\ref{Thm2.1}, that the Milnor fibers are homotopy equivalent to 
compact classical symmetric spaces of Cartan; and besides obtaining 
their cohomology we give their homotopy groups in a stable range using 
Bott periodicity.  As well, certain of the Milnor fibers exhibit 
$2$-torsion in their cohomology.  We further show that for either class, 
excluding the symmetric $m \times m$ matrices with $m$ even, the 
monodromy acts trivially on the (rational or) complex cohomology of the 
Milnor fiber.  For the links in either case, we compute the cohomology 
vector space as a truncated and shifted exterior algebra and determine 
that its cohomology product structure is \lq\lq almost trivial\rq\rq, 
Theorems~\ref{Thm2.2} and \ref{Thm3.6}.  
\par
For the complex cohomology of the complement for reductive groups these 
results extend those obtained in Granger-Mond et al. \cite{GMNS}, although 
in that paper they also prove the cohomology is computed by the 
logarithmic complex, which we do not.  The results here also extend the 
results obtained for the complement and Milnor fiber in the case of 
solvable algebraic groups in Damon-Pike \cite{DP}.  
\par
Lastly, we combine in \S 6 the results on exceptional orbit hypersurfaces 
we have described together with a Theorem of Mutsuo Oka \cite{Ok} and 
the results of Siersma et al to compute the topology of a class of 
hypersurface singularities which are formed as \lq\lq formal sums\rq\rq 
of hypersurface singularities each of which is either an exceptional orbit 
hypersurface which we consider or is a  weighted homogeneous 
nonisolated hypersurface singularity considered by the Siersma group.  
These yield  hypersurface singularities whose Milnor fibers are homotopy 
equivalent to joins of compact manifolds, or more generally to a bouquet 
of suspensions of such joins of compact manifolds. 
\par
We would like to thank both David Mond and Shrawan Kumar for their input 
from several valuable conversations on these questions, and the referee 
for a number of very useful suggestions. \par

\section{Prehomogeneous Spaces and the Wang Sequence}  
\label{S:sec1}
\par
\subsection*{A Special Class of Prehomogeneous Spaces}
\par
We consider a special class of representations of a connected complex 
linear algebraic groups $\rho : G \to \GL(V)$ which have an open orbit 
$\cU$.  We specifically consider the cases where the {\em exceptional 
orbit variety}, which is the union of the orbits of positive codimension, is 
a hypersurface $\cE$.  If $H$ is the isotropy subgroup for a point $v_0 \in 
\cU$, then $H$ is a closed algebraic subgroup of $G$, and it is a basic fact, 
see e.g. Borel \cite{Bo2}, that both $G$ and $H$ have maximal compact 
subgroups $K$, resp. $L$, with $L \subset K$, which are strong 
deformation retracts of $G$, resp. $H$, and of the same ranks as $G$ and 
$H$. \par
 For example this will include the cases where $V = M$ is one of the 
spaces of complex matrices $M = Sym_m$ or $M = Sk_m$ (for $m = 2k$) 
acted on by $\GL_m(\C)$ by $B\cdot A = B A B^T$, or , $M = M_{m, m}$ and 
$\GL_m(\C)$ acts by left multiplication.  Each of these representations 
have open orbits and the resulting prehomogeneous space has an 
exceptional orbit variety $\cE$ which is a hypersurface. 
\begin{Definition} 
\label{Def1.1}
The {\em determinantal hypersurface} for the space of $m \times m$ 
symmetric or general matrices, denoted by $M = Sym_m$ or $M = M_{m, 
m}$ is the hypersurface of singular matrices defined by $\det : M \to \C$ 
and denoted by $\cD_m^{sy}$ for $M = Sym_m$, or $\cD_m$ for $M = M_{m, 
m}$.  For the space of $m \times m$  skew-symmetric matrices $M = 
Sk_m$ (for $m = 2k$) the determinantal hypersurface of singular matrices 
is defined by the Pfaffian $\Pf : Sk_m \to \C$, and is denoted by 
$\cD_m^{sk}$.  
\end{Definition}
\par
A second class of examples consists of \lq\lq equidimensional 
representations\rq\rq where $\dim_{\C} V = \dim_{\C} G$ with an open 
orbit $\cU$.  Then, necessarily $\cE$ is a hypersurface and the isotropy 
subgroup $H$ of a point in the open orbit is finite.  If an appropriate 
defining equation for $\cE$, obtained from the coefficient determinant of 
the associated vector fields for the action, is reduced then $\cE$ is a 
\lq\lq linear free divisor\rq\rq, introduced by Mond and Buchweitz 
\cite{BM}; and otherwise it is a slightly weaker linear free* divisor with 
nonreduced defining equation.  Partial results on the topology of the 
complement were obtained if $G$ is reductive, see e.g. \cite {GMNS}) and 
also for the complement and Milnor fiber for \lq\lq block 
representations\rq\rq for $G$ solvable (see \cite{DP}, \cite{DP2}).  We 
shall determine the topology in the general case.  
\par
For any of these cases, the action of $G$ commutes with the usual
$\C^*$-action on $V$; hence, the exceptional orbit variety $\cE$ is also 
invariant, and hence has a homogeneous reduced defining equation $f$ of 
degree $n$.  In the case of equidimensional representations  there is a 
defining equation given by the coefficient determinant of the associated 
vector fields of degree $N = \dim_C V = \dim_C G$.  If $\cE$ is a linear 
free divisor, this is a reduced definng equation, which we may choose for 
$f$; if not then $\deg f = n < N$.  \par
We consider the Milnor fibration of the hypersurface germ $(\cE, 0)$ in the 
standard form as $f^{-1}(S^1_{\gd}) \cap D^{2N}_{\gevar}$ with 
$D^{2N}_{\gevar}$ the disk about $0$ of sufficiently small radius 
$\gevar$, $S^1_{\gd}$ the boundary of the disk about $0$ in $\C$ of radius 
$\gd$, for $0 < \gd < < \gevar$.  Because of the homogeneity of $f$ we may 
adapt a standard argument for isolated singularities to obtain the global 
description of the Milnor fibration.  \par
\begin{Lemma}
\label{Lem1.1}
The Milnor fibration of $(\cE, 0)$ is diffeomorphic to the fibration $f| E : E 
\to S^1$, where $E = f^{-1}(S^1)$, with fiber $F = f^{-1}(1)$.  This is the 
restriction of the fibration $f: V \backslash \cE \to \C^*$ to $S^1 \subset 
\C^*$, and the inclusion $E \subset V \backslash \cE$ is a homotopy 
equivalence.  
\end{Lemma}
\begin{proof}
The proof is a slight modification of that for the case of isolated  
singularities.  It uses the induced homogeneous $\R_+$-action to establish 
a diffeomorphism of fibrations between $E = f^{-1}(S^1)$ and 
$f^{-1}(S^1_{\gd})$.  If $0 < \gd << \gevar$ are sufficiently small, it is 
proven that the restriction of the distance-squared function $\| \cdot\|^2$ 
to $f^{-1}(S^1_{\gd}) \backslash D^{2N}_{\gevar}$ has no critical points, 
and by Morse theory we deduce that the Milnor fibration is diffeomorphic 
as a fibration to $f^{-1}(S^1_{\gd})$.  For the case of equidimensional 
representations of solvable linear algebraic groups the details are given 
in the proof of \cite[Thm 3.2]{DP}.  The general case follows the same line 
of reasoning.  
\end{proof}
We refer to the fibration $F \hookrightarrow E \to S^1$ as the {\em global 
Milnor fibration} and $F$ as the {\em global Milnor fiber}.  We note that by 
e.g. Kato-Masumoto \cite{KM}, provided $N \geq 2$, $F$ is path-connected.
\par
Second, the cohomology of the complement $V \backslash \cE$ (and hence 
$E$), can be expressed in terms of the maximal compact subgroups.
\begin{Lemma}
\label{Lem1.2}
If $G$ and the isotropy subgroup $H$ of $v_0 \in \cU$ have maximal 
compact subgroups $K$, resp. $L$, then $V \backslash \cE$ is homotopy 
equivalent to $K/L$.  Hence, 
$$ H^*(V \backslash \cE) \,\, \simeq H^*(K/L) \, . $$
\end{Lemma}
\begin{proof}
 First, the action of $G$ on $v_0 \in \cU$, gives the open orbit $V 
\backslash \cE  = \cU \simeq G/H$ as a homogenenous space.  It is 
sufficient to show the inclusion $K/L \hookrightarrow G/H$ is a homotopy 
equivalence.  As pointed out by Shrawan Kumar, this is a simple 
consequence of the long exact homotopy sequence for a fibration.  The 
groups $G$, resp. $H$, are homotopy equivalent to their maximal compact 
subgroups, $K$, resp.$L$.  We consider the long exact homotopy sequence 
of the fibrations (see e.g. \cite[Chap 7, \S 2, Thm 10]{Sp}) in the 
horizontal rows of (\ref{CD1.3}). 
\begin{equation}
\label{CD1.3}
\begin{CD} 
 {L} @>>>  {K} @>>> {K/L}    \\
 @V{\iti^{\,\prime}}VV     @V{\iti}VV      @V{\iti_0}VV   \\
 {H} @>>>  {G}  @>>>  {G/H}  
\end{CD}  
\end{equation}
which has the form
\begin{equation}
\label{CD1.3a}
\begin{CD} 
{  } @>>>  {\pi_j(L, Id)} @>>>  {\pi_j(K, Id)} @>>> {\pi_j(K/L, \{L\})} @>>> { }  
\\
 @.    @V{\iti^{\,\prime}_*}VV     @V{\iti_*}VV      @V{\iti_{0\, *}}VV   @. 
\\
{  } @>>>  {\pi_j(H, Id)} @>>>  {\pi_j(G, Id)} @>>> {\pi_j(G/H, \{H\})} @>>> { } 
\end{CD}
\end{equation}  
We note that when $j = 0$, both $\iti_*$ and $\iti^{\,\prime}_*$ are 
bijections of pointed sets so $K$ is connected and $H$ and $L$ have the 
same finite number of connected components.  The homotopy equivalence 
$\iti^{\,\prime}: L \hookrightarrow H$ restricts to a homotopy equivalence 
$\iti^{\,\prime}: L_0 \hookrightarrow H_0$ of the connected components 
of the identity.  Thus, $\iti^{\,\prime}_* : \pi_j(L, Id) = \pi_j(L_0, Id) 
\simeq \pi_j(H_0, Id) = \pi_j(H, Id)$ is an isomorphism for all $j \geq 0$.  
Hence, we can apply the five Lemma to conclude first that $\pi_i(K/L, 
\{L\}) \simeq \pi_i(G/H, \{H\})$ for all $i \geq 0$.  As both are 
path-connected CW-complexes, it follows by Whitehead\rq s theorem that 
$K/L$ is homotopy equivalent to $G/H$, giving the result.
\end{proof} 
\par 
Next, we consider the global Milnor fibration $F \hookrightarrow E \to 
S^1$.  First, from the long exact homotopy sequence for the fibration, we 
obtain $\pi_i(F) \simeq \pi_i(E)$ for $i \geq 2$ and, as $F$ is 
path-connected, the short exact sequence 
\begin{equation}
\label{CD1.4}
\begin{CD} 
 {1} @>>> {\pi_1(F)} @>>> {\pi_1(E)} @>p>> {\pi_1(S^1)}  @>>>  {1} \, .
\end{CD}  
\end{equation}
Hence, by the Hurewicz theorem and the universal coefficient theorem, 
$p_* : H_1(E) \to H_1(S^1)$ is surjective, and $p^* : H^1(S^1; \bk) \to 
H^1(E; \bk)$ is injective for any field $\bk$ of characteristic $0$, so the 
image of a generator of $H^1(S^1; \bk)$ is a nonzero class $s_1 \in 
H_1(E)$.  Furthermore, this class restricts to $0$ in $H^1(F; \bk)$, so then 
does the ideal $\langle s_1\rangle \cdot H^*(E; \bk)$. \par
\begin{Example}
\label{Ex1.3}
For the case of complex cohomology, if $\gg$ is a generator of 
$\pi_1(S^1)$, then there is a $\gb \in \pi_1(E)$ such that $f_*(\gb) = \gg$.  
Hence, $\frac{df}{f} = f^*(\frac{dz}{z})$ is a closed $1$-form which 
satisfies
$$  \int_{\gb} \frac{df}{f} \,\, = \,\, \int_{f_*(\gb)} \frac{dz}{z} \,\, = \,\,  
2\pi i \, .$$
Hence ,$\gw_1 = \frac{df}{f}$ defines a nontrivial cohomology class $s_1$ 
in $H^1(E; \C)$.  Moreover, its restriction to $F$ is $0$ as $df = 0$ as $f 
\equiv 1$ on $F$.  Thus, the ideal $\C\langle s_1\rangle \cdot H^*(E; \C)$ 
restricts to $0$ in $H^*(F; \C)$.  
\end{Example}
\par
\subsection*{Model Complex Geometry for the Milnor Fiber} 
We note that although by Lemma \ref{Lem1.1} the Milnor fiber and global 
Milnor fiber are diffeomorphic, they are not holomorphically 
diffeomorphic.  The different complex structure on the global Milnor fiber 
allows us to introduce an alternate way to view the topology on the global 
Milnor fiber resulting from it having a form of \lq\lq complex model  
geometry\rq\rq in the sense of Thurston (see e.g. \cite[Chap. 3]{Th}).  By 
this we mean a smooth affine submanifold $X \subset \C^N$ together with 
an algebraic action of a connected linear algebraic group $G$ on $X$, which 
is transitive with $H$ an isotropy subgroup of a point.  We denote it by the 
triple $(X, G, H)$.  It provides a method for beginning to understand the 
geometry of the global Milnor fiber.  Actually, Thurston\rq s definition of 
geometry over the reals required $X$ be simply connected, which we refer 
to it as a {\em simple model geometry}.   However, without this 
restriction, the following is true.
\begin{Proposition}
\label{Prop1.1}
For a prehomogeneous space defined by the representation $\rho : G \to 
\GL(V)$ with exceptional orbit variety a hypersurface, there is a 
connected codimension one algebraic subgroup $G^{\prime}_0$ of $G$ 
which acts transitively on the global Milnor fiber $F$, with isotropy 
subgroup $H^{\prime}$ so that $(F, G^{\prime}, H^{\prime})$ defines a 
model complex geometry on $F$.  This model is simple in the case of 
determinantal hypersurfaces.  
\end{Proposition}
The proof of this will follow from results in \S~\ref{S:sec2}.  
This model geometry type representation of the global Milnor fiber is the 
essential ingredient, in place of Morse theory, for the analysis of the 
topology.   To carry out this analysis, we will make considerable use of 
the Wang cohomology sequence to relate the cohomology of $E$ with that 
of $F$. 
\vspace{1ex}
\subsection*{The Wang Sequence} 
\par
We consider a fibration $F \hookrightarrow E \overset{p}\rightarrow S^1$.  
We let $\gs = \gs_1$ be the monodromy map where $\gs_t$ denotes the 
lift of the curve  $\varphi(t) = e^{2\pi i t}$ on $S^1$. to the family of 
homeomorphisms  $\gs_t : F \to F_t = p^{-1}(\varphi(t))$.  
\begin{Definition}
We say that the fibration $F \hookrightarrow E \to S^1$ is {\em 
(rationally) cohomologically trivial} if $\gs^* : H^*(F; \Q) \to H^*(F; \Q)$ is 
the identity.
\end{Definition}
\begin{Remark}
\label{Rem1.1} 
By the universal coefficient theorem it likewise follows that for any field 
$\bk$ of characteristic $0$, $\gs^* = id$ on $H^*(F; \bk)$ and conversely if 
it holds for $\bk$ of characteristic $0$ then it also holds for $\Q$.
  Also, if $H^*(F; \Z)$ is a free abelian group then it also follows that the 
Milnor fibration is cohomologically trivial for integer coefficients.  
\par
We also might ask whether the monodromy is geometrically trivial in that 
it is homotopy equivalent to the identity.  David Mond has been able to 
show directly for several examples in the equidimensional case, including 
certain quiver representations, that the monodromy is geometrically 
trivial; however, for certain other equidimensional cases we discussed 
neither of us were unable to verify this.  For now there is no general 
result for geometric triviality of monodromy, even for certain special 
classes such as quiver representations, so it is still an open question.  
However, we will see that being cohomologically trivial will suffice for 
many questions.  
\end{Remark}
To begin, there is the following result.
\begin{Proposition}
\label{Prop1.5}
For the fibration $F \hookrightarrow E \to S^1$, with $F$ connected and 
$\bk$ a field of characteristic $0$, the following are equivalent.
\begin{itemize}
\item[i)] The fibration is cohomologically trivial.
\item[ii)]  there is an isomorphism of graded vector spaces.
\begin{align}
\label{Eqn1.5}
 H^*(E; \bk) \,\, &\simeq \,\,  \gL^* \bk\langle s_1\rangle \otimes H^*(F; 
\bk) \notag \\ 
 &\simeq \,\,  H^*(F; \bk) \,\,\, \oplus \,\,\, \bk\langle s_1\rangle 
\otimes H^*(F; \bk)\, . 
\end{align}
\item[iii)] 
\begin{equation}
\label{Eqn1.5b}
 \dim_{\bk} H^*(E; \bk) \,\, = \,\, 2 \,\dim_{\bk} H^*(F; \bk)\, .
\end{equation}
\end{itemize}
\par 
Moreover, if the preceding hold then (\ref{Eqn1.5}) is an isomorphism of 
graded $\gL^* \bk\langle s_1\rangle$-modules, where the exterior algebra 
$\gL^* \bk\langle s_1\rangle$ is on one generator $s_1$.  
\end{Proposition}
\begin{proof}
\par
First, we immediately observe that ii) imples iii). 
Second, suppose i) holds.  We consider the Wang sequence in cohomology 
(see e.g. \cite[Chap. 8 \S 5 Cor. 6]{Sp}) applied to our situation.  

\begin{multline}
\label{CD1.4a}
\begin{CD}
@>>> {H^q(F; \bk)} @>{\theta}>> {H^q(F; \bk)} @>>> {H^{q+1}(E; \bk)} 
\end{CD} \\ 
\begin{CD}
@>{\iti^*}>> {H^{q+1}(F; \bk)}  @>{\theta}>>  {H^{q+1}(F; \bk)} @>>> \, 
\end{CD}  
\end{multline}
where $\theta = id -\gs^*$ for $\gs$ the monodromy map.  
\par 
If the fibration is cohomologically trivial, then $\gs^* = id$, and  
(\ref{CD1.4a}) reduces to the short exact sequences for $q \geq -1$,
\begin{equation}
\label{CD1.5}
\begin{CD} 
 {0} @>>> {H^q(F; \bk)} @>>> {H^{q+1}(E; \bk)} @>{\iti^*}>> {H^{q+1}(F; \bk)} 
@>>> {0} \, 
\end{CD}
\end{equation}
It then follows that the Betti numbers, which equal  
$\gb_{q} = \dim_{\bk} H^q( \,\cdot \, ; \bk)$, satisfy $\gb_{q+1}(E) = 
\gb_{q}(F) + \gb_{q+1}(F)$  for $q \geq -1$ (where $\gb_{-1} = 0$).  This 
says that there is a graded vector space isomorphism (\ref{Eqn1.5}) and 
also that (\ref{Eqn1.5b}) holds.  \par
Third, if iii) holds we show that i) holds.  
For any $q \geq 0$, let $K_q = \ker(I - \gs^*) : H^q(F; \bk) \to H^q(F; \bk)$ 
and $C_q = \coker(I - \gs^*)$.  Then, $\dim_{\bk} K_q = \dim_{\bk} C_q$ 
and from (\ref{CD1.4a}) 
\begin{equation}
\label{Eqn1.6}
\gb_{q+1}(E) \,\, = \,\, \dim_{\bk} K_{q+1} + \dim_{\bk} C_q \,\, = \,\, 
\dim_{\bk} K_{q+1} + \dim_{\bk} K_q \, .
\end{equation}  
Hence, if we choose an $r > \dim F$ so that $H^j(F; \bk) = 0$ for $j \geq r$, 
then summing (\ref{Eqn1.6}) over $q$ yields
\begin{align}
\label{Eqn1.7}
  \dim_{\bk} H^*(E; \bk) \,\, &= \,\,  \sum_{q = -1}^{r-1}\dim_{\bk} 
K_{q+1} \,\,  +\,\,  \sum_{q = 0}^{r}\dim_{\bk} K_q \notag   \\
&= \,\,   2 \sum_{q = 0}^{r}\dim_{\bk} K_q
\end{align}
Hence, if the fibration is not cohomologically trivial then for some $q$, 
$\dim_{\bk} K_q < \dim_{\bk} H^q(F; \bk)$.  It follows that the RHS of 
(\ref{Eqn1.7}) is less than $2 \,\dim_{\bk} H^*(F; \bk)$, so iii) doesn\rq t 
hold. 
\par
Lastly, if the above hold, then (\ref{Eqn1.5}) is an isomorphism of graded 
vector spaces, and (\ref{CD1.4a}) reduces to the short exact sequences  
(\ref{CD1.5}).  Thus, for each $q$, we can find a set of elements $\{ 
\varphi^{(q)}_1, \dots, \varphi^{(q)}_{m_q}\}$ in $H^q(E; \bk)$ which 
restrict to a basis for $H^q(F; \bk)$.  Then, we can apply the Leray-Hirsch 
Theorem, see e.g. \cite[Chap. 5, \S 7,Thm 9]{Sp}, to conclude that $H^*(E; 
\bk)$ is a free 
$H^*(S^1; \bk)$-module on this set of generators.  As $H^*(S^1; \bk) \simeq 
\gL^*\bk\langle s_1\rangle$, we conclude that  (\ref{Eqn1.5}) is an 
isomorphism of graded $\gL^*\bk\langle s_1\rangle$-modules.
\end{proof}
\par

\vspace{2ex}
 \subsection*{Topology of the Complement, Milnor Fiber, and Link} 
\par
We consider a representation $\rho : G \to \GL(V)$ which belongs to the 
special class of prehomogeneous spaces: with open orbit $\cU$, isotropy 
subgroup $H$ of $v_0 \in \cU$, with maximal compact subgroups $K$, resp. 
$L$, exceptional orbit variety $\cE$ which is a hypersurface, and whose 
global Milnor fiber is cohomologically trivial.  We apply the preceding 
results to draw conclusions about the cohomology of the complement, the 
Milnor fiber and link of $\cE$ using the preceding notation.  \par
We have already considered both the cohomology of the complement and 
Milnor fiber.  Since $\cE$ is invariant under the usual $\C^*$-action, we 
may use the restricted  $\R_+$-action to conclude that $L(\cE) =  \cE \cap 
S^{2N-1}$ is diffeomorphic to the link of $\cE$, where $S^{2N-1}$ denotes 
the unit sphere about the origin in $V$.  \par 
To describe the cohomology of the link we use the following notation. For 
a compact connected orientable manifold $X$, we introduce the graded 
vector space 
$\widetilde{H^*(X; \bk)}\left[ r\right]$ which will denote the vector space 
$H^*(X; \bk)$, truncated at the top degree and shifted upward by degree 
$r$. 
Then we can summarize the topology of the exceptional orbit variety 
$\cE$ by the following. 
\begin{Proposition}
\label{Prop1.8}
Suppose $\rho : G \to \GL(V)$ is a representation which belongs to the 
special class of prehomogeneous spaces.  
\begin{itemize}
\item[i)] {\em Cohomology of Milnor fiber $F$:  }  \par
If the  global Milnor fibration is cohomologically trivial, then
$$   H^*(F; \bk) \,\, \simeq   H^*(E; \bk)/ \left(\bk\langle s_1 \rangle 
\smile H^*(E; \bk)\right)  $$
where as graded vector spaces, 
$$   H^*(E; \bk) \,\, \simeq   H^*(F; \bk) \,\, \oplus \,\, \left(\bk\langle 
s_1 \rangle \otimes H^*(F; \bk)\right) \, ; $$
\item[ii)]  {\em Cohomology of the Complement $V \backslash \cE$:  }
$$   H^*(V \backslash \cE; \bk) \,\, \simeq   H^*(K/L; \bk) \, ; $$
\item[iii)]  {\em Cohomology of the Link $L(\cE)$: } \par
If $K/L$ is orientable, then 
$$   \widetilde{H}^*(L(\cE); \bk) \,\, \simeq   \widetilde{H^*(K/L; 
\bk)}\left[ 2N - 2 - \dim_{\R} K/L\right] \, ; $$
\end{itemize}
\end{Proposition}
\begin{proof}
\par
We have already discussed i) in Proposition \ref{Prop1.5}, and ii) in 
Lemma \ref{Lem1.2}.  \par 
For the link, as $\cE$ is homogeneous (of degree $n$), the complement $V 
\backslash \cE$ is also invariant under the usual $\C^*$ action.  Thus, we 
may contract using the $\R_+$-action to conclude that 
$\cE \backslash \{ 0\}$ is homotopy equivalent to the link $L(\cE)$ and 
$V \backslash \cE$ is homotopy equivalent to the complement of the link 
$S^{2N-1} \backslash L(\cE)$.  Thus, 
\begin{equation}
\label{Eqn1.9}
 H^*(S^{2N-1} \backslash L(\cE); \bk) \,\, = \,\, H^*(V \backslash \cE; \bk) 
\, . 
\end{equation}
Second, we apply Alexander duality for subspaces of spheres (see e.g. 
\cite[Chap. XIV, Thm 6.6]{Ma}) to conclude for all $j$
\begin{equation}
\label{Eqn1.10}
  \widetilde{H}^j(L(\cE); \bk) \,\, \simeq \,\, \widetilde{H}_{2N-2-
j}(S^{2N-1} \backslash L(\cE); \bk) \, . 
\end{equation}
Hence, combined with (\ref{Eqn1.9}), we obtain for all $j$
\begin{equation}
\label{Eqn1.11}
\widetilde{H}^j(L(\cE); \bk)\,\, \simeq \,\, \widetilde{H}_{2N-2-j}(V 
\backslash \cE; \bk) \, . 
\end{equation}
By Lemma \ref{Lem1.2} 
\begin{equation}
\label{Eqn1.12}
H^{*}(V \backslash \cE; \bk) \,\, \simeq \,\, H^{*}(K/L; \bk)
\end{equation}
If $q = \dim_{\R} K/L$, then, as $K/L$ is compact and orientable, by 
Poincar\'{e} duality $H^{q-r}(K/L; \bk) \simeq H^{r}(K/L; \bk)$ for all $0 
\leq r \leq q$.  
Thus, for $j = 2N-2-(q-r)$, we obtain from (\ref{Eqn1.11})  for $0 < r < q$
\begin{equation}
\label{Eqn1.13}
\widetilde{H}^{2N-2-q+r}(L(\cE); \bk) \,\, \simeq \,\, \widetilde{H}_{q-
r}(K/L; \bk) \,\, \simeq \,\, \widetilde{H}^{q-r}(K/L; \bk)  \,\, \simeq \,\, 
\widetilde{H}^{r}(K/L; \bk)  \, .
\end{equation}
If $r = q$ we obtain
$$ \widetilde{H}^{2N-2)}(L(\cE); \bk) \,\, \simeq \,\, 
\widetilde{H}^{0}(K/L; \bk) \,\, = \,\,  0\, .$$
This last equation also follows as $\dim_{\R}L(\cE) = 2N-3$.  If  $j = 0$ in 
(\ref{Eqn1.11}) and $N > 2$ then as $2N-2 > N = \dim_{\R}K/L$, 
$$ \widetilde{H}^{0}(L(\cE); \bk) \,\, \simeq \,\, \widetilde{H}^{2N-2}(K/L; 
\bk) \,\, = \,\,  0\, .$$
Combining (\ref{Eqn1.13}) and the comment following it with 
(\ref{Eqn1.12}), yields the result if $N > 2$.  \par 
The only cases when $N \leq 2$ are $Sym_1$, $M_{1, 1}$, and $Sk_2$, each 
of which is dimension $1$ with $\GL_1(\C) = C^*$, $K = S^1$, $H = \{ 1\}$, 
and $\cE = \{ 0\}$.  In each case the Milnor fiber is a single point, the link 
is empty, and the complement is homotopy equivalent to $S^1$.  Then, iii) 
has to be understood as correct for $\widetilde{H}^j(L(\cE); \bk)$ for $j 
\geq 0$ and then the result is true.  
\end{proof}
\par
Then, we will use these results to determine when the global Milnor 
fibrations arising from two classes of exceptional orbit hypersurfaces are 
cohomologically trivial, and show that the cohomology $H^*(K/L; \bk)$ is 
either an exterior algebra or a module over an exterior algebra. \par 

\section{Modeling Global Milnor Fibers for Exceptional Orbit 
Hypersurfaces}  
\label{S:sec2}
\par
Let $\rho : G \to GL(V)$ be a representation of a connected linear algebraic 
group $G$ defining a prehomogeneous space with exceptional orbit variety 
$\cE$ a hypersurface.  We will identify a subgroup of $G$ which defines a 
model complex geometry on the global Milnor fiber $F$.  Using this we 
relate the corresponding maximal connected compact subgroups via an 
exact sequence that gives a fibration which we show is always 
cohomologically trivial.  
\par
For now we consider the general case with $\cE$ a hypersurface with 
reduced homogenous defining equation $h = 0$.  We first define a 
representation $\chi$ of $G$ on $\C\langle h \rangle$ as follows. If $g \in 
G$, then we let  $g^*h(v) = h(g^{-1}\cdot v)$ for $v \in V$.  As the action of 
$g^{-1}$ preserves the orbits of $G$, $g^*h$ still vanishes on $\cE$.  Thus, 
by the Nullstellensatz, $g^*h$ is a multiple of $h$.  Also, as the action of 
$g$ is given by a linear transformation of $V$, $g^*h$ is again a 
polynomial of the same degree as $h$; hence, $g^*h = c_g\cdot h$, for $c_g 
\in \C$.  This defines a representation  $\chi_0 : G \to \GL(\C\langle h 
\rangle) \simeq \C^*$.  We let $G^{\prime} = \ker(\chi_0)$ and 
$G^{\prime}_0$ denote the connected component of the identity of 
$G^{\prime}$.  We begin with a simple Lemma.

\begin{Lemma}
\label{Lem1b.0}
In the preceding situation, $\chi_0$ is non-trivial and induces by a lifting 
an exact sequence
\begin{equation}
\label{CD1b.1}
\begin{CD} 
 {1} @>>> {G^{\prime}_0} @>>> {G} @>{\chi}>> {\C^*} @>>> {1} \, 
\end{CD}
\end{equation}
where $G^{\prime}$ and $G^{\prime}_0$ are linear algebraic groups with 
$\dim_{\C} G^{\prime}_0 = \dim_{\C} G - 1$ and $\rank (G^{\prime}_0) = 
\rank (G) - 1$.  Moreover, if $G$ is reductive, respectively solvable, then 
so is $G^{\prime}_0$ reductive, respectively solvable. 
\end{Lemma}
\begin{proof} \par
First, if $\chi_0$ were trivial, then  $G$ acts trivially on $h$; hence, it 
will preserve the fibers $h^{-1}(w)$ for any $w \in \C$.  This says that the 
orbits of $G$ are all of positive codimension, so there is no open orbit, a 
contradiction.  Thus, $\chi_0$ is surjective and $G/G^{\prime} \simeq 
\C^*$. \par
Since $G^{\prime}$ is a closed subgroup of $G$, it is an algebraic subgroup, 
as is the subgroup $G_0^{\prime}$.  Also, $G^{\prime}$ and $G^{\prime}_0$ 
have the same dimension and rank.  From $G/G^{\prime} \simeq \C^*$ we 
see $\dim_{\C} G^{\prime}_0 = \dim_{\C} G^{\prime} =\dim_{\C} G - 1$.  
\par 
 As $G^{\prime}$ is algebraic, it has only finitely many connected 
components.  Thus, the quotient map $G/G^{\prime}_0 \to G/G^{\prime} 
\simeq \C^*$ is a finite covering map of $\C^*$ and hence $G/G^{\prime}_0 
\simeq \C^*$.  Thus, we have the exact sequence given in (\ref{CD1b.1}), 
where $\chi$ denotes the lifting of $\chi_0$ for the covering map.  \par
Next, if $T$ denotes a maximal algebraic torus of $G$ of dimension $k = 
\rank(G)$, then $\chi_0 | T$ is still onto $\C^*$, otherwise its image 
would be $\{ 1\}$.  Since every element in $G$ is conjugate to an element 
of $T$, $\chi$ would be trivial on $G$, which as we just saw is 
impossible.  Thus, $\ker(\chi_0 | T)$ is a torus of dimension $k - 1$. Thus, 
$\rank(G^{\prime}) \geq k-1$.  If $\rank(G^{\prime}) = k$, then there is a 
maximal algebraic torus $T^{\prime} \subset G^{\prime}$ of rank $k$.  
Then, $T^{\prime}$ is also a maximal algebraic torus of $G$ and $\chi_0 | 
T^{\prime}$ is trivial, a contradiction.  Thus, $\rank(G^{\prime}) = k - 1$. 
\par
If $G$ is solvable, then so is its subgroup $G^{\prime}_0$.  If $G$ is 
reductive, and $K$ is the maximal compact subgroup with Lie algebra $\k$, 
then $\g = \k + i \k$.  If $\tilde \chi : \g \to \C$ is the Lie algebra 
homomorphism associated to $\chi$, then as it is $\C$ linear, it is the 
complexification of $\tilde \chi | \k$.  Hence, the Lie algebra of 
$G^{\prime}$, $\g^{\prime} = \ker(\tilde \chi)$ is the complexification of 
$\ker(\tilde \chi) = \k^{\prime}$, the Lie algebra of $K^{\prime}$.  Thus, if 
$G^{\prime}_0$ and $K^{\prime}_0$ denote the corresponding connected 
components of the identity, then as $K^{\prime}_0$ is compact, and 
$\k^{\prime}$ a real form for $\g^{\prime}$, then by \cite[Thm 2.4.7, Prop. 
2.4.2]{GW}, $G^{\prime}_0$ and then $G^{\prime}$ are reductive. \par
This completes the proof.

\end{proof}
\par  
Second, we consider the associated exact sequence of maximal compact 
subgroups.  For the maximal compact subgroup $K \subset G$, $\chi(K)$ is 
a connected compact subgroup of $\C^*$, and hence is either $\{ 1\}$ or 
$S^1$.  Again if $\chi(K) = \{ 1\}$, let $T_0$ be a maximal torus of $K$ 
contained in a maximal algebraic torus $T$ of $G$ and homotopy 
equivalent to it.  Since $\chi_0 | T_0$ is trivial, it follows that $\chi | T$ 
is trivial, which is a contradiction.  
\par 
We let $\chi^{\prime} = \chi | K$ and $K^{\prime} = \ker(\chi^{\prime})$, 
$G^{\prime} \cap K$.   We also let $K^{\prime}_0$ be the connected 
component of the identity of $K^{\prime}$, so $K^{\prime}_0 \subset 
G^{\prime}_0$.  \par
Then, we have the following diagram with rows exact, and the vertical 
arrows are inclusions.  
\begin{equation}
\label{CD1b.3}
\begin{CD} 
{1} @>>>  {G^{\prime}_0} @>>> {G} @>{\chi}>> {\C^*}  @>>> {1}  \\
@.   @A{\iti^{\,\prime}}AA     @A{\iti}AA      @A{\iti_0}AA  @. \\
{1} @>>>  {K^{\prime}_0} @>>>  {K}  @>{\chi^{\prime}}>>  {S^1} @>>> {1} 
\end{CD}  
\end{equation}
We remark that the bottom row is exact, because the covering 
$K/K^{\prime}_0 \to K/K^{\prime} \simeq S^1$ is the restriction of the 
covering for $G/G^{\prime}_0 \to G/G^{\prime}$. 
\par
There are three things that we want to show in this situation:
\begin{itemize}
\item[i)]  $\iti^{\,\prime}$ is a homotopy equivalence;
\item[ii)] $G^{\prime}_0$ acts transitively on the global Milnor fiber $F$; 
and 
\item[iii)] the fibration $K^{\prime}_0 \hookrightarrow K \to S^1$ is 
cohomologically trivial.
\end{itemize}

\begin{Lemma}
\label{Lem1b.2}
In the above situation in the diagram (\ref{CD1b.3}), $\iti^{\,\prime}$ is a 
homotopy equivalence.  
\end{Lemma}
\begin{proof}
Each of the rows of (\ref{CD1b.3}) gives fibrations:  $G^{\prime}_0 
\hookrightarrow G \to \C^*$ and $K^{\prime}_0 \hookrightarrow K \to 
S^1$. Since $K$ is the maximal compact subgroup of $G$, $\iti : K 
\hookrightarrow G$ is a homotopy equivalence, as is the inclusion $\iti_0 
: S^1 \hookrightarrow \C^*$. Then, we may consider the induced maps for 
the long exact sequences in homotopy for the two fibrations, and again by 
the 5-lemma, since both $\iti$ and $\iti_0$ are homotopy equivalences, it 
follows that $\iti^{\,\prime}_*: \pi_j(K^{\prime}_0) \simeq 
\pi_j(G^{\prime}_0)$ for all $j \geq 0$.  Since they are both 
path-connected CW-complexes, it follows by Whitehead\rq s theorem that 
$\iti^{\,\prime}$ is a homotopy equivalence.  
\end{proof}
\par
Next, as $G^{\prime}_0$ acts trivially on $h$, it acts on the fibers of $h$, 
and in particular on the global Milnor fiber $F$.  Let $H^{\prime}$ be the 
isotropy subgroup of $G^{\prime}_0$ for a point $v_0$ in $F$, and let 
$L^{\prime} \subset K^{\prime}_0$ be the maximal compact subgroup of 
$H^{\prime}$.  
\par
\begin{Lemma}
\label{Lem1b.3}
In the preceding situation, $G^{\prime}_0$ acts transitively on the global 
Milnor fiber and $K^{\prime}_0/L^{\prime} \subset F$ is a deformation 
retract.  In the equidimensional case, $H^{\prime}$ is finite and hence 
equals $L^{\prime}$  and $G^{\prime}_0$ is a finite regular covering space 
of $F$ with group of covering transformations $H^{\prime}$.  
\end{Lemma}
\begin{proof}
\par
By Lemma \ref{Lem1b.0}, $\dim_{\C} G^{\prime}_0 = \dim_{\C} G - 1$ and 
it acts on the global Milnor fiber $F$ of $\dim_{\C} F = N - 1$, where 
$\dim_{\C} V = N$.  If $v_0 \in F$, then as $v_0 \in \cU$, the orbit map 
$\Phi : G \to V$ sending $g \mapsto g\cdot v_0$ is a local submersion at 
$Id$.  Hence, if $d\Phi(Id) | T_{Id} G^{\prime}_0$ has rank $< N-1$, then 
$d\Phi(Id)$ has rank $< N$, contradicting it being a submersion.  Thus, 
$d\Phi(Id) | T_{Id} G^{\prime}_0$ has rank $N-1$, and as $G^{\prime}_0$ 
preserves the fibers of $h$, $\Phi(G^{\prime}_0) \subset F$, so its image 
is a neighborhood of $v_0$.  As this is true at each point $v_0 \in F$,  The 
orbit of $v_0$ under $G^{\prime}_0$ is open in $F$.  Thus, any orbit of 
$G^{\prime}$ in $F$ is open in $F$.  As $F$ is connected (by e.g. 
Kato-Matsumoto \cite{KM}), there can only be one orbit and $G^{\prime}_0$ 
acts transitively on $F$.  Hence, if $H^{\prime}$ is the isotropy subgroup 
in $G^{\prime}_0$ of $v_0$, $F$ is diffeomorphic to 
$G^{\prime}_0/H^{\prime}$.  Also, by the same argument given in Lemma 
\ref{Lem1.2}, $K^{\prime}_0/L^{\prime} \subset G^{\prime}_0/H^{\prime}$ 
is a deformation retract, establishing the first result.  \par
In the equidimensional case, $\dim_{\C} G^{\prime}_0 = \dim_{\C} F = N - 
1$.  As they have the same dimension, the isotropy subgroup of a point is a 
zero dimensional algebraic subgroup, and hence finite.  As $G^{\prime}_0$ 
is connected and  $F \simeq G^{\prime}_0/H^{\prime}$, $G^{\prime}_0$ is a 
covering space as stated. \par
\end{proof}
\par
\begin{Remark}
\label{Rem1.2}
\par
We see now that Proposition \ref{Prop1.1} follows from Lemma 
\ref{Lem1b.3}, since $(F, G^{\prime}_0, H^{\prime})$ gives the model 
complex geometry for $F$.  For the determinantal hypersurfaces, the 
singular sets for $\cD^{sy}_m$, $\cD_m$, and $\cD^{sk}_m$ ($m$ even) 
have codimension $\geq 2$, so by Kato-Matsumoto \cite{KM} the Milnor 
fibers are simply connected, so the model is simple. \par
We also remark that Lemma \ref{Lem1b.3} gives an alternate approach to 
the topological structure of the Milnor fiber. Rather than trying to 
represent it as being homotopy equivalent to a bouquet of spheres, it gives 
a compact submanifold which is a deformation retract.
\end{Remark}
We complete this section by establishing for a fibration of connected 
compact Lie groups that cohomological triviality always hold.  
\begin{Lemma}
\label{Lem1b.4}
For an exact sequence of compact connected Lie groups 
\begin{equation}
\label{CD1b.4}
\begin{CD} 
{1} @>>>  {K^{\prime}_0} @>>>  {K}  @>{\chi^{\prime}}>>  {S^1} @>>> {1} 
\end{CD}  
\end{equation}
The associated fibration $K^{\prime}_0 \hookrightarrow K \to S^1$ is 
cohomologically trivial.
\end{Lemma}
\begin{proof}
First, by the same arguments as in Lemma \ref{Lem1b.0} but for compact 
groups, $\dim_{\R} K^{\prime}_0 = \dim_{\R} K - 1$ and $\rank 
K^{\prime}_0 = \rank K - 1 = k - 1$.  
\par 
Thus, by the basic Hopf structure theorem, the cohomology with 
coefficients in a field $\bk$ of characteristic $0$ of a compact connected 
Lie group of rank $k$ is a Hopf algebra which is isomorphic to an exterior 
algebra on $k$ generators of odd degrees (see e.g. Borel \cite{Bo} which 
treats the more general Leray Structure theorem for Hopf algebras, where 
$\bk$ may be a perfect field).  Thus, $H^*(K; \bk)$ is isomorphic to an 
exterior algebra on $k$ generators and $H^*(K^{\prime}_0; \bk)$ is 
isomorphic to an exterior algebra on $k-1$ generators.  Thus, $\dim_{\bk} 
H^*(K; \bk) = 2^k$, and $\dim_{\bk} H^*(K^{\prime}_0; \bk) = 2^{k-1}$.  
Hence, $\dim_{\bk} H^*(K; \bk) = 2\,\dim_{\bk} H^*(K^{\prime}_0; \bk)$, so 
by Proposition \ref{Prop1.5}   the fibration is cohomologically trivial. 
\end{proof}
\par
This last result by itself will not imply that the global Milnor fibration is 
cohomologically trivial.  In fact, we see in the next section that it fails 
for a family of determinantal hypersurfaces.  However, the preceding will 
provide the crucial pieces in \S \ref{S:sec5} to prove that for any 
prehomogeneous space defined from an equidimensional representation 
$\rho : G \to \GL(V)$ of a connected linear algebraic group $G$, the global 
Milnor fibration is cohomologically trivial. 

\section{Topology of Determinantal Hypersurfaces}  
\label{S:sec3}
\par
	We first consider the determinantal hypersurfaces which arise as 
the  exceptional orbit varieties of the representations of $\GL_m(\C)$ on  
$M = Sym_m$, $M_{m, m}$, or $Sk_m$ (for $m = 2k$), given before 
Definition \ref{Def1.1}.  We separately state the results for the topology 
of the Milnor fibers, and the complements and links. \par
\vspace{2ex}

\par
\vspace{1ex}
\begin{Thm} 
\label{Thm2.1}
The Milnor fibers of the determinantal hypersurfaces are homogeneous 
spaces homotopy equivalent to compact symmetric spaces of \lq\lq 
classical type\rq\rq \, (classified by Cartan) given in Table 
\ref{Mil.fib.sym.sp}.  The cohomology with coefficients in a field $\bk$ of 
characteristic $0$ of the Milnor fiber is either an exterior algebra, or a 
free module on two genersators over an exterior algebra, as given in Table 
\ref{Mil.fib.sym.sp}.  Also, the Milnor fibration is cohomologically trivial 
for all cases except the case of $m \times m$ symmetric matrices with 
$m$ even.  \par 
Moreover, for the regular and skew-symmetric cases the formulas remain 
true when $\bk$ is replaced by $\Z$, while for the symmetric case there 
is $2$-torsion in the Milnor fiber and the mod-$2$ cohomology of the 
Milnor fiber is given by the $\Z_2$-exterior algebra
$$   H^*(SU_m/SO_m(\R); \Z_2) \,\, = \,\, \gL^*\, \Z_2\langle s_2, s_3,  
..., s_m\rangle $$  
with $s_j$ of degree $j$.  
\end{Thm}
\par
\vspace{2ex}
\begin{table}[h]
\begin{tabular}{|l|c|c|l|}
\hline
Determinantal  & Milnor Fiber  & Symmetric   & $H^*(F,\bk)$  \\
  Hypersurface &  F   &  Space    &    \\
\hline
$\cD_m^{sy}$ & $SL_m(\C)/SO_m(\C)$ & $SU_m/SO_m$  &  
$\gL^*\bk\langle e_5, e_9, \dots , e_{2m-1}\rangle$\\
(m = 2k+1)  &     &     &    \\
\hline
$\cD_m^{sy}$ & $SL_m(\C)/SO_m(\C)$ & $SU_m/SO_m$  &  
$\gL^*\bk\langle e_5, e_9, \dots , e_{2m-3}\rangle\cdot\{ 1, e_{m}\}$  \\
(m = 2k)  &     &     &     \\
\hline
$\cD_m$  & $SL_m(\C)$ & $SU_m$  &  $\gL^*\bk\langle e_3, e_5, \dots , 
e_{2m-1}\rangle$\\
\hline
$\cD_m^{sk}, m = 2k$ &  $SL_{2k}(\C)/Sp_{k}(\C)$  & $SU_{2k}/Sp_k$  &  
$\gL^*\bk\langle e_5, e_9, \dots , e_{2m-3}\rangle$  \\
\hline
\end{tabular}
\caption{The generators of the cohomology $e_k$ are in degree $k$; and 
the structure is an exterior algebra for all cases except for 
$\cD_{2k}^{sy}$, for which the structure is a free module on $1$ and 
$e_{m}$ over the exterior algebra.}
\label{Mil.fib.sym.sp}
\end{table}
Along with the topology of the Milnor fiber, we can also give the topology 
of the complement and the (co)homology of the link in the next theorem. 
\vspace{1ex}
\begin{Thm} 
\label{Thm2.2}
The complements $M \backslash \cD$ of the determinantal hypersurfaces 
$\cD$ are homogeneous spaces $G/H$ which are homotopy equivalent to 
quotients of maximal compact subgroups $K/L$, given in column 2 of Table 
\ref{Coh.Compl.Lnk}.  The cohomology with coefficients in a field $\bk$ of 
characteristic $0$ of the complement is an exterior algebra  
as given in column 3 of Table \ref{Coh.Compl.Lnk}.  
Moreover, for the regular and skew-symmetric cases the formulas remain 
true when $\C$ is replaced by $\Z$.  
\par  
The cohomology of the link $L(\cD)$ is isomorphic as a graded vector 
space to the cohomology of the complement truncated in the top degree 
and shifted in degree.  The shift is $N - 2$ where $N = \dim_{\C} M$ for all 
cases except the symmetric case with $m$ even; and in that case the shift 
is $N - 2 + m$.  Then the products in the reduced cohomology 
$\widetilde{H}^*(L(\cD); \C)$ are $0$ except possibly for the single 
product $e_1^* \smile 1^*$, where the $e_1^*$ and $1^*$ are the images in 
degrees $N-1$, resp. $N-2$, of $e_1, 1 \in H^*(M \backslash \cD; \C)$ in 
all cases  except for the symmetric case with $m$ even. 
\end{Thm}
\par
\vspace{2ex}
\begin{table}
\begin{tabular}{|l|c|c|l|}
\hline
Determinantal  & Complement  & $H^*(M \backslash \cD,\bk) \simeq$   & 
\,\,Shift  \\
  Hypersurface &  $M \backslash \cD$   &  $H^*(K/L,\bk)$   &   \\
\hline
$\cD_m^{sy}$ & $GL_m(\C)/O_m(\C)$ &  $\gL^*\bk\langle e_1, e_5, \dots , 
e_{2m-1}\rangle$  & $\binom{m+1}{2} - 2$ \\  
(m = 2k+1)  & $\sim U_m/O_m(\R)$    &    &  
\\
\hline
$\cD_m^{sy}$ & $GL_m(\C)/O_m(\C)$ & $\gL^*\bk\langle e_1, e_5, \dots , 
e_{2m-3}\rangle$  &  $\binom{m+1}{2} + m - 2$ \\
(m = 2k)  &     &     & 
\\
\hline
$\cD_m$  & $GL_m(\C) \sim U_m$ & $\gL^*\bk\langle e_1, e_3, \dots , 
e_{2m-1}\rangle$  &  $m^2 - 2$  \\
\hline
$\cD_m^{sk}$ &  $GL_{2k}(\C)/Sp_{k}(\C)$  & $\gL^*\bk\langle e_1, e_5, 
\dots , e_{2m-3}\rangle$  &  $\binom{m}{2} - 2$  \\
(m = 2k) &  $\sim U_{2k}/Sp_{k}$   &     &    \\
\hline
\end{tabular}
\caption{The cohomology of the complements $M \backslash \cD$ and links 
$L(\cD)$ for each determinantal hypersurface $\cD$.  The complements, 
are homotopy equivalent to the quotients of maximal compact subgroups 
$K/L$ with cohomology given in the third column, where the generators of 
the cohomology $e_k$ are in degree $k$; and the structure is an exterior 
algebra.  For the links $L(\cD)$, the cohomology is isomorphic as a vector 
space to the cohomology of the complement truncated in the top degree 
and shifted by the degree indicated in the last column.}
\label{Coh.Compl.Lnk}
\end{table}

\begin{Example}[Simplest Determinantal Hypersurfaces which are Morse 
Singularities]
\label{Exam2.1}
We consider the complex cohomology for the simplest examples in Table 
\ref{Mil.fib.sym.sp} of dimension $> 1$ and clarify the notation.  First, for 
$\cD^{sy}_2$ we have $m = 2$ with $k = 1$.  As $4k - 3 = 1 < 5$, the 
exterior algebra is on $0$ generators.  By this we mean it is just $\C 
\cdot 1$, with $1$ the identity in the cohomology algebra.  The Milnor 
fiber then has complex cohomology $\C \langle 1, e_2 \rangle$.  As 
$\cD^{sy}_2$ is defined by an equation of the form $xz -y^2 = 0$, we see it 
is a Morse singularity with cohomology as stated.  In addition, in Table 
\ref{Coh.Compl.Lnk}, the complement has cohomology $\C \langle 1, e_1 
\rangle$; and after truncating $e_1$ and shifting by $3$, we obtain that 
the reduced complex cohomology of the link $L(\cD^{sy}_2)$, which is $\R 
P^3$, is nonzero only in degree $3$, where it is $\C$.  \par
Next $\cD_2$ is defined by an equation of the form $xw -yz = 0$.  Again it 
is a Morse singularity and the table with $m = 2$ gives for the complex 
cohomology of the Milnor fiber $\gL^* \C \langle e_3\rangle  = \C \langle 
1, e_3\rangle$.  The complement has cohomology $\gL^* \C\langle e_1, 
e_3\rangle  =  \C \langle 1, e_1, e_3, e_1e_3\rangle$, which when 
truncated and shifted gives the reduced complex cohomology of the link 
$L(\cD_2)$ as being $\C$ in degrees $2$, $3$ and $5$ and zero otherwise.  
Similarly, $\cD^{sy}_4$ is defined by the Pfaffian of the form $xw -yv + 
zu = 0$, which is again a Morse singularity and the Table with $m = 2$ 
gives as the complex cohomology of the Milnor fiber  $\gL^* \C\langle 
e_5\rangle  =  \C \langle 1, e_5\rangle$; and the reduced complex 
cohomology of the link $L(\cD^{sk}_2)$ is $\C$ in degrees $4$, $5$ and $9$ 
and zero otherwise.  Because in the later two cases the link is a compact 
oriented manifold,  we see that the one possible nonzero product $e_1^* 
\smile 1^*$ is indeed nontrivial. \par 
\end{Example}
\par
\begin{Example}
\label{Exam2.2}
 Beyond the cases in Example \ref{Exam2.1}, the classical approach of 
representing the Milnor fiber as a bouquet of spheres no longer applies as 
the cohomology algebra would be trivial with no torsion.  For example, for 
$\cD^{sy}_3$, the Table gives the complex cohomology as $\gL^* \C \langle 
e_5\rangle  =  \C \langle 1, e_5\rangle$; however, we also have for 
$\Z_2$-cohomology, $\gL^* \Z_2 \langle s_2, s_3\rangle = \Z_2 \langle 1, 
s_2, s_3, s_2 s_3\rangle$, which gives $2$-torsion in degrees $2$ and 
$3$ and a non-trivial product structure.  Likewise, other higher 
determinantal hypersurfaces have Milnor fibers not homotopy equivalent 
to a bouquet of spheres.  \par
The cohomology of the complement $Sym_3 \backslash \cD^{sy}_3$ is the 
exterior algebra $\gL^* \C \langle e_1, e_5\rangle  = \C \langle 1, e_1, 
e_5, e_1 e_5\rangle$.  Also, $N-2 = 6-2 = 4$, so the link $L(\cD^{sy}_3)$ 
has reduced complex cohomology $\C \langle 1^*, e_1^*, e_5^*\rangle$ 
with degrees $4, 5, 9$.  where the product $e_1^* \smile 1^*$ still has to 
be determined.  
\end{Example}
\par
\subsection*{Stable Homotopy Groups of the Milnor Fibers}
Third, using the periodicity Theorems of Bott, we can also give the stable 
homotopy groups of the Milnor fibers up to the appropriate stable range.  
Again this is because we have already identified them as being homotopy 
equivalent to the classical symmetric spaces, so the results of Bott, as 
applied to these spaces, yield the calculations, with it only being 
necessary to observe where the stable range begins in each case.  For the 
three cases, we the corresponding infinite dimensional symmetric spaces 
are given by: \par
$$\mbS\mbU = \cup_{n = 1}^{\infty} SU_n, \quad \mbS\mbU/\mbS\mbO = 
\cup_{n = 1}^{\infty} SU_n/SO_n, \quad \text{ and } \quad  
\mbS\mbU/\mbS p = \cup_{n = 1}^{\infty} SU_{2n}/Sp_n \, .$$  
We also let the Milnor fibers of the determinantal hypersurfaces be 
denoted by: $F_m$, $F^{sy}_m$, and $F^{sk}_m$.  Then, the homotopy groups 
of the Milnor fibers can be given up to the end of the stable range in terms 
of the stable homotopy groups of the infinite dimensional symmetric 
spaces as follows.
\par
\begin{Thm}
\label{Thm2.3}
The homotopy groups of the Milnor fibers up to the end of the stable range 
are as follows.
\begin{itemize}
\item[i)]
$$ \pi_j(F_m) \,\, \simeq \,\,  \pi_j(SU_m)  \,\, \simeq \,\, 
\pi_j(\mbS\mbU) \qquad \text {for }  j < 2m    $$
\item[ii)]
$$ \pi_j(F^{sy}_m) \,\, \simeq \,\,  \pi_j(SU_m/SO_m)  \,\, \simeq \,\, 
\pi_j(\mbS\mbU / \mbS\mbO) \qquad \text {for }  j < m -1    $$
\item[ii)] for $m = 2k$
$$ \pi_j(F^{sk}_m) \,\, \simeq \,\,  \pi_j(SU_{2k}/Sp_k)  \,\, \simeq \,\, 
\pi_j(\mbS\mbU / \mbS p) \qquad \text {for }  j < 4k -2    $$
\end{itemize}
where the stable homotopy groups are given in Table \ref{Stab.htpy.grps}.
\end{Thm}
\vspace{2ex}
\begin{table}[h]
\begin{tabular}{|l|c|c|c|c|c|c|c|c|c|l|}
\hline
$\pi_j(\mbG/\mbH)$&  &  &   &    &  &  &   &    &  &  \\
 $j =  $  & 0 & 1  & 2   & 3  & 4  & 5   & 6 & 7  & 8  & 9 \\
\hline
$\mbS\mbU $  & $0$ & $0$ & $0$  &  $\Z$ & $0$ & $\Z$  & $0$  & $\Z$   & 
$0$  & $\Z$ \\
\hline
$\mbS\mbU / \mbS\mbO$  & $0$  & $0$ & $\Z_2$  &  $\Z_2$ & $0$  & $\Z$   
& $0$ & $0$ & $0$ & $\Z$\\
\hline
$\mbS\mbU/\mbS p$  & $0$ & $0$ & $0$ & $0$ & $0$ & $\Z$  & $\Z_2$  &  
$\Z_2$ & $0$  & $\Z$ \\
\hline
\end{tabular}
\caption{The stable homotopy groups of the infinite dimensional 
symmetric spaces $\mbG/\mbH$.  They are periodic of period dividing $8$, 
with the periodicity beginning at $j = 2$.}
\label{Stab.htpy.grps}
\end{table}
\par
We note that for the Milnor fibers $F_m$, $m \geq 6$, $F^{sy}_m$, $m \geq 
11$, and $F^{sk}_m$, $m \geq 6$, exhibit all of the homotopy groups 
appearing in Table \ref{Stab.htpy.grps} in their stable ranges. 
\par
\subsection*{Proofs of the Theorems} \hfill 
\par  
\begin{proof}[Proof of Theorem \ref{Thm2.1}]
We first observe that in all three cases the determinantal hypersurface is 
homogeneous defined by either $\det^{-1}(0)$ for $M = Sym_m$ or $M = 
M_{m, m}$ or $\Pf^{-1}(0)$ for $M = Sk_m$ (for $m = 2k$).  By Lemma 
\ref{Lem1.1}, we may consider the  Milnor fibers of the global Milnor 
fibration.  \par 
The simplest case is for the Milnor fiber for the determinantal 
hypersurface $\cD_m$ for $m \times m$ matrices $M_{m, m}$, which is 
just $F = \det^{-1}(1) = SL_m(\C)$.  It is homotopy equivalent to its 
maximal compact subgroup $SU_m$, and so, for example, by 
\cite[Chap. 3, Thm 6.5]{MT},its cohomology is a Hopf algebra given by the   
exterior algebra 
$$   H^*(SU_m; \Z) \,\,  \simeq \,\,  \gL^* \Z \langle e_3,  \dots , e_{2m -
1}\rangle \, .  $$
 where on the $e_i$ are of degree $i$ (which correspond by transgression 
homomorphisms to the Chern classes).  Furthermore, by replacing $\Z$ by 
$\bk$, a field of characteristic $0$,  we obtain the corresponding result 
for cohomology with coefficients in $\bk$.  \par
The second case $\cD_m^{sy}$ is for the $m \times m$ symmetric case $M 
= Sym_m$.  We claim the Milnor fiber $F = \det^{-1}(1)$, is diffeomorphic 
to $SL_m(\C)/SO_m(\C)$.  First, the action of $SL_m(\C)$ on $F$ by $A 
\mapsto BAB^T$ is transitive.  We know by the classification of 
symmetric matrices under the equivalence $A \mapsto BAB^T$, that if 
$\det(A) \neq 0$, then there is a $B \in GL_m(\C)$ such that $BAB^T = 
I_m$, the $m \times m$ identity matrix.  If $\det(B) = b \neq 0$ and 
$\det(A) = 1$, then  $\det(B)\det(A)\det(B^T) = b^2 = 1 = \det(I_m)$.  
Hence, if we replace $B$ by $B^{\prime} = b^{-1}B$, then 
$B^{\prime}AB^{\prime\, T} = b^{-2}\cdot I_m = I_m$, and 
$\det(B^{\prime}) = 1$.  Thus, the orbit of $I_m$ under $SL_m(\C)$ is $F$. 
\par
Moreover, the isotropy subgroup of $I_m$ under the action of $SL_m(\C)$ 
is $\{ B \in SL_m(\C): B\,I_m\,B^T = I_m\}$, which is $SO_m(\C)$.  Hence, 
$F \simeq SL_m(\C)/SO_m(\C)$.  Lastly, the groups $SL_m(\C)$, resp. 
$SO_m(\C)$, are homotopy equivalent to their maximal compact 
subgroups, which are $SU_m$, resp. $SO_m(\R)$.  Hence, by the argument 
given in Lemma \ref{Lem1.2}, $SL_m(\C)/SO_m(\C)$ is homotopy 
equivalent to $SU_m/SO_m(\R)$, which is one of the classical symmetric 
spaces (see \cite[Chap. 3, \S 6]{MT}).  Thus, $H^*(SL_m(\C)/SO_m(\C)) 
\simeq H^*(SU_m/SO_m(\R))$.  The calculation of $H^*(SU_m/SO_m(\R))$ 
is given in e.g. \cite[Chap. 3, Thm 6.7]{MT} depends on whether $m$ is even 
or odd.  If $m$ is odd $ = 2k+1$, then for a field $\bk$ of characteristic 
$0$, it is an exterior algebra 
$$   H^*(SU_m/SO_m(\R); \bk) \,\, = \,\, \gL^* \bk\langle e_5, e_9,  ..., 
e_{2m -1}\rangle \, . $$ 
If $m$ is even $ = 2k$, then it is a free module of rank $2$ over an 
exterior algebra
$$   H^*(SU_m/SO_m(\R); \bk)  \,\, = \,\,  \gL^* \bk\langle e_5, e_9,  ..., 
e_{2m -3}\rangle \{1, e_m\}  $$
In the last case $e_m^2 \neq 0$, but instead equals an expression 
involving products of the odd degree generators.. 
\par
The remaining case is for the $m \times m$ skew-symmetric case $M = 
Sk_m$ with $m = 2k$.  In this case the global Milnor fiber 
$F = \Pf^{-1}(1)$ is diffeomorphic to $SL_m(\C)/Sp_m(\C)$.  \par
First, the action of $SL_m(\C)$ on $F$ by $A \mapsto BAB^T$ is transitive.  
We know by the classification of skew-symmetric matrices under the 
equivalence $A \mapsto BAB^T$, that if $\det(A) \neq 0$, then there is a 
$B \in GL_m(\C)$ such that $BAB^T = J_m$, where $J_m$ is the $2 \times 
2$ block diagonal matrix with $2 \times 2$ diagonal blocks 
$\begin{pmatrix} 0 & 1 \\ -1 & 0 \end{pmatrix}$.  Again if $\det(B) = b 
\neq 0$ and $\Pf(A) = 1$, then  $\Pf(BAB^T) = \det(B) \Pf(A) = b = 1 = 
\Pf(J_m)$.  Hence, $B \in SL_m(\C)$; thus, the orbit of $J_m$ under 
$SL_m(\C)$ is $F$.   The isotropy group of $J_m$ in $SL_m(\C)$ is $\{ B 
\in SL_m(\C) : B J_m B^T = J_m\} = Sp_k(\C)$.  Thus, as before,  $F \simeq 
SL_{2k}(\C)/Sp_k(\C)$.  \par 
Lastly, the groups $SL_{2k}(\C)$, resp. $Sp_k(\C)$, are homotopy 
equivalent to their maximal compact subgroups, which are $SU_{2k}(\C)$, 
resp. $Sp_k$.  We conclude $SL_{2k}(\C)/Sp_k(\C)$ is homotopy equivalent 
to $SU_{2k}(\C)/Sp_k$, which is again one of the classical symmetric 
spaces of compact type (see \cite[Chap. 3, \S 6]{MT}).  \par
Then, by \cite[Chap. 3, Thm 6.7]{MT}
$$   H^*(SU_{2k}(\C)/Sp_k; \Z)  = \gL^* \Z\langle e_5, e_9,  ..., e_{2m -
3}\rangle $$
where again $e_i$ has degree $i$ and $m = 2k$.  We obtain the 
corresponding result with $\Z$ replaced by $\bk$.
 \par
For the symmetric case, we again apply \cite[Chap. 3, Thm 6.7]{MT} to 
conclude that the mod-$2$ cohomology is as claimed. \par
The only remaining claim is that the Milnor fibrations are cohomologically 
trivial except in the case of $m \times m$ symmetric matrices with $m$ 
even.  This will follow from the results in Theorem \ref{Thm2.2} together 
with Proposition \ref{Prop1.5}.
\end{proof}
\begin{proof}[Proof of Theorem \ref{Thm2.2}]
The proof will follow in each case from Proposition \ref{Prop1.8}.  \par 
In the case of $M_{m, m}$, the complement is $\GL_m(\C)$, which has 
maximal compact subgroup $U_m$ with cohomology a Hopf algebra given 
by the exterior algebra 
$$   H^*(U_m; \Z) \,\,  \simeq \,\,  \gL^* \Z \langle e_1, e_3,  \dots , 
e_{2m -1}\rangle \, ,  $$
and replacing $\Z$ by $\bk$ we obtain the corresponding result for 
cohomology with $\bk$ coefficients. \par
Second, in the skew-symmetric case  $M = Sk_m$ with $m = 2k$, the 
isotropy subgroup of $\GL_m(\C)$ for $J_m$ is $\{ B \in \GL_m(\C) : B 
J_m B^T = J_m\} = Sp_m(\C)$.  They have maximal compact subgroups 
$U_m$, resp. $Sp_k(\R)$.  Second by \cite[Chap. 3, Thm 6.7]{MT}, 
$$   H^*(U_{2k}(\C)/Sp_k; \Z)  = \gL^* \Z\langle e_1, e_5, e_9,  ..., e_{2m -
3}\rangle \, ,$$
and we can again replace $\Z$ by $\bk$ to obtain the corresponding result 
for cohomology with coefficients in $\bk$.  \par
Third, the symmetric case is the most subtle.  We will verify it for 
complex cohomology and then the result will follow for any field $\bk$ of 
characteristic $0$.  For $M = Sym_m$, the isotropy subgroup of 
$\GL_m(\C)$ for $I_m$ is $\{ B \in \GL_m(\C) : B I_m B^T = I_m\} = 
O_m(\C)$.  They have maximal compact subgroups $U_m$, resp. $O_m(\R)$; 
and $O_m(\R)$ has two connected components with $SO_m(\R)$ being the 
connected component of the identity.  Thus, first by Proposition 
\ref{Prop1.8}, $H^*(M\backslash \cE; \C) \simeq H^*(U_m/O_m(\R); \C)$.  
Second by \cite[Chap. 3, Thm 6.7]{MT}, we may compute the cohomology 
$H^*(U_m/SO_m(\R); \C)$, which depends on whether $m$ is even or odd.  
If $m$ is odd $ = 2k+1$, then it is an exterior algebra 
\begin{equation}
\label{Eqn2.8}
   H^*(U_m/SO_m(\R); \C) \,\, = \,\, \gL^* \C\langle e_1, e_5, e_9,  ..., 
e_{2m -1}\rangle \, .
\end{equation} 
If $m$ is even $ = 2k$, then it is a free module of rank $2$ over an 
exterior algebra
\begin{equation}
\label{Eqn2.9}
  H^*(U_m/SO_m(\R); \C)  \,\, = \,\,  \gL^* \C \langle e_1, e_5, e_9,  ..., 
e_{2m -3}\rangle \{1, e_m\} \, . 
\end{equation}
To obtain the results for $H^*(U_m/O_m(\R); \C)$ from 
$H^*(U_m/SO_m(\R); \C)$ we use that $p: U_m/SO_m(\R) \to 
U_m/O_m(\R)$ is a double covering obtained from $U_m/SO_m(\R)$ as a 
quotient of the action of $O_m(\R)/SO_m(\R) \simeq \Z_2$.  This action 
is given by a covering transformation $\tau$ of $U_m/SO_m(\R)$ defined 
as follows.  We  let $C_m$ denote the diagonal matrix which equals $1$ 
except for the last entry, which is $-1$.  Then, $\tau(A\cdot SO_m(\R)) = 
A\cdot C_m \cdot SO_m(\R)$.  This is well-defined as $SO_m(\R)$ is 
normal in $O_m(\R)$ so if $A^{\prime} = A\cdot D$ with $D \in SO_m(\R)$, 
then $A^{\prime}\cdot C_m = A\cdot D\cdot C_m = A\cdot C_m \cdot 
D^{\prime}$ with $D^{\prime} = \cdot C_m^{-1} \cdot D \cdot C_m \in 
SO_m(\R)$.  Then, $U_m/O_m(\R)$ is the resulting quotient. \par

We obtain by a standard result, see Lemma \ref{Lem3.3}, 
\begin{equation}
\label{Eqn2.10}
 H^*(U_m/O_m(\R); \C) \,\, \simeq \,\, H^*(U_m/SO_m(\R); \C)^{\Z_2}\, .
\end{equation}
The subtle point of the calculation is to compute the RHS of 
(\ref{Eqn2.10}).  Then, the result for the symmetric case follows from the 
following Lemma.  
\begin{Lemma}
\label{Lem2.12}
The action of $\tau^*$ on $H^*(U_m/O_m(\R); \C)$ is given by 
$\tau^*(e_{4i+1}) = e_{4i+1}$ for all $i$, and if $m$ is even, 
$\tau^*(e_{m}) = - e_{m}$
\end{Lemma}
It follows from the Lemma that in the odd dimensional case $m = 2k +1$ 
the entire cohomology is invariant under $\tau^*$.  By contrast, in the 
even dimensional case $m = 2k$, the exterior algebra generated by all of  
the odd degree generators $e_{4i+1}$ is invariant under $\tau^*$; however 
the elements in the module on $e_m$ over the exterior algebra are all sent 
to their negatives.  Thus, the invariant cohomology under $\tau^*$ is 
exactly the exterior algebra as stated.  Before proving the Lemma, which 
gives the result for the symmetric case, we give the remaining argument 
for the link.  \par
By Proposition \ref{Prop1.8} we obtain the cohomology as the truncated 
and shifted cohomology of the complement.  If $K/L$ is orientable then the  
degree is shifted by $2N - 2 -\dim_{\R} K/L$.  With the sole exception of 
the $m \times m$ symmetric matrices with $m$ even, for the 
determinantal hypersurfaces, $K/L$ is orientable.  For example, we can 
see this because the cohomology computed in Table \ref{Coh.Compl.Lnk} is 
nonzero in the degree $= \dim_{\R} K/L$.  However, for $m \times m$ 
symmetric matrices with $m$ even, this is not true, so $K/L$ is not 
orientable so we will consider this case separately.  \par
Then, for the other cases, the complement is a quotient of reductive 
groups $G/H$ with $N = \dim_{\C} (G/H) = \dim_{\R} K/L$, so the shift is 
given by $N - 2$.  For the symmetric case with $m$ even, the cohomology 
is still an exterior algebra, but top nonzero cohomology occurs in degree 
$m$ below the top degree ($ N = \dim_{\R} K/L$), if we follow the 
argument in Proposition \ref{Prop1.8}, we see that the shift must be 
altered to $N -2+m$, giving the result as claimed.  \par
The properties of the cohomology product follow by considering degrees.  
We have $\dim_{\R} L(\cD) = 2N-3$, $H^{2N-4}(L(\cD); \C) = 0$ (as 
$H^{2}(K/L; \C) = 0$) in each case.  Excluding the symmetric case with $m$ 
even, the lowest positive degree classes have degrees $N-2$ for $1^*$ and 
$N-1$ for $e_1^*$, then the lowest dimensional products in the reduced 
cohomology are given by: $1^* \cup 1^*$ of degree $2N-4$ and hence $0$, 
and $e_1^* \cup 1^*$ of degree $2N-3$.  All other products have higher 
degree and hence are $0$.  For the symmetric case with $m$ even, the 
lowest degree term $1^*$ has degree $N - 2+ m$, so the lowest degree 
product will have degree $2(N - 2+ m) = 2N -3 + (2m - 1) > 2N -3$ so all 
products are zero in this case. \par
Lastly, we return to the question of the Milnor fibration being 
cohomologically trivial.  Since $H^*(E; \C) \simeq H^*(M \backslash \cE; 
\C)$, the calculations of the cohomology given by Theorems~\ref{Thm2.1} 
and \ref{Thm2.2} (and the universal coefficient theorem) imply in each 
case except for the symmetric case with $m$ even, that (\ref{Eqn1.5b}) is 
satisfied, so by Proposition \ref{Prop1.5}, the Milnor fibration is 
cohomologically trivial.  In the symmetric case with $m = 2k$ even, from 
Table \ref{Mil.fib.sym.sp} we see the cohomology of the Milnor fiber is a 
free module on two generators over an exterior algebra on $k-1$ 
generators, and has total dimension $2^k$; while from Table 
\ref{Coh.Compl.Lnk} the cohomology of the total space is an exterior 
algebra on $k$ generators and has the same total dimension $2^k$.  Thus, 
(\ref{Eqn1.5b}) does not hold, so the Milnor fibration for these cases is not 
cohomologically trivial. 
\end{proof}
\begin{proof}[Proof of Lemma \ref{Lem2.12}]
\par  
We let $\tau^{\prime}$ denote the diffeomorphism of $U_n$ defined by 
$\tau^{\prime}(A) = A\cdot C_m$.  Then, $\tau^{\prime}$ descends in the 
quotient map $p : U_n \to U_n/O_n(\R)$ to yield $\tau$.  Furthermore, by 
\cite[Chap. 3, \S6, Thm 6.7]{MT}, the map $p^*$ in complex cohomology 
uniquely sends $p^*(e_{4j+1}) = \tilde e_{4j+1}$ in $H^*(U_n; \C)$ for all 
$j$ with $e_{4j+1}$ nonzero in $H^*(U_n/O_n(\R); \C)$.  Here we use 
$\tilde e_{2j+1}$ to denote the exterior algebra generators of $H^*(U_n; 
\C)$.  Since $U_n$ is connected, we can choose a path from $C_m$ to $Id$, 
and this defines a homotopy between $\tau^{\prime}$ and the identity map 
on $U_n$.  Thus, $\tau^{\prime\, *}(\tilde e_{4j+1}) = \tilde e_{4j+1}$ for 
all $j$.  Then, 
$$  p^*(\tau^*(e_{4j+1})) \,\, = \,\, \tau^{\prime\, *}(p^*(e_{4j+1})) \,\, = 
\,\, \tau^{\prime}*(\tilde e_{4j+1}) \,\, = \,\, \tilde e_{4j+1} \,\, = \,\, 
p^*(e_{4j-3})\, .$$  
By the uniqueness of $p^*(e_{4j+1})$, we conclude $\tau^*(e_{4j+1}) = 
e_{4j+1}$ as claimed.
\par
It remains to consider in the case $m = 2k$, $\tau^*(e_{m})$ which is the 
Euler class $e(E)$ of the $m$ dimensional bundle $E$ over 
$U_m/SO_m(\R)$ defined from the standard representation of $SO_m(\R)$ 
on $\R^m$.  We consider the fibration $p^{\prime} : U_m/T^k \to 
U_m/SO_m(\R)$ where $T^k$ is the $k$-torus $SO_2(\R) \times \cdots 
\times SO_2(\R)$ (with $k$ factors).  Then $p^{\prime\, *} : 
H^*(U_m/SO_m(\R); \C) \to H^*(U_m/T^k; \C)$ is injective.  Also, the 
pull-back $p^{\prime\, *}(E)$ splits into oriented real $2$-plane bundles 
(or using $SO_2(\R) \simeq U_1$, complex line bundles) $L_1 \oplus L_2 
\oplus \cdots \oplus L_k$.  Then $\tau^*(e(E)) = e(\tau^*E)$.  Since $C_m$ 
is in the normalizer of $T^k$, we can define a diffeomorphism 
$\tau^{\prime\prime}$ of $U_m/T^k$ by $\tau^{\prime\prime}(A\cdot T^k) 
= (A\cdot C_m\cdot T^k)$.  This is seen to be well-define just as was 
$\tau^{\prime}$.  \par 
Then, $p^{\prime\, *}(\tau^*(e_m)) = p^{\prime\, *}(\tau^*(e(E))) 
=\tau^{\prime\prime\, *}(p^{\prime\, *}(e(E)))$.  Also, by the splitting, 
$p^{\prime\, *}(e(E)) = \prod_{j = 1}^{k} e(L_j)$.  Thus, 
$\tau^{\prime\prime\, *}(p^{\prime\, *}(e(E))) = \prod_{j = 1}^{k} 
\tau^{\prime\prime\, *}e(L_j)$.  Now, the effect of $\tau^{\prime\prime}$ 
on the first $k-1$ factors of $T^k$ is as the identity so 
$\tau^{\prime\prime\, *}e(L_j) = e(L_j)$ for $j < k$.  On the last factor 
$\tau^{\prime\prime}$ acts by multiplication by the reflection matrix 
$C_2 = \begin{pmatrix} 1 & 0 \\ 0 & -1 \end{pmatrix}$.  This changes the 
orientation of $L_k$ and hence changes the sign of $e(L_k)$.  \par 
Thus, $\tau^{\prime\prime\, *}e(L_k) = -e(L_k)$.  Hence, 
$\tau^{\prime\prime\, *}(p^{\prime\, *}(e(E))) = - e(E)$.  Finally, from the 
above we conclude 
$$  p^{\prime\, *}(\tau^*(e(E))) \,\, = \,\, \tau^{\prime\prime\, 
*}(p^{\prime\, *}(e(E))) \,\, = \,\, - p^{\prime\, *}(e(E))\, .$$
Since $p^{\prime\, *}$ is injective, we finally conclude $\tau^*(e_m) = 
\tau^*(e(E)) = - e(E) = - e_m$, as claimed.
\end{proof}
\par
\begin{proof}[Proof of Theorem \ref{Thm2.3}]
Lastly, for Theorem \ref{Thm2.3}, we begin by noting that the groups in 
Table \ref{Stab.htpy.grps} are obtained from Bott periodicity, see 
\cite[Chap. 4, \S 6,Table 4.1]{MT}.  Then, as we already know that the 
Milnor fibers are homotopy equivalent to the specific classical symmetric 
spaces, it is only necessary to see that the stable ranges are as stated.  
This follows by standard type arguments applying the homotopy long exact 
sequence to the fibrations $SU_n \hookrightarrow SU_{n+1} \to S^{2n + 
1}$, resp. $SO(n)(\R) \hookrightarrow SO_{n+1}(\R) \to S^{n}$, and $Sp_n 
\hookrightarrow Sp_{n+1} \to S^{4n - 1}$, together with $SO
_n(\R) \hookrightarrow SU_{n} \to SU_n/SO_n(\R)$, and $Sp_n 
\hookrightarrow SU_{2n} \to SU_{2n}/Sp_n(\R)$ to obtain
\begin{itemize}
\item[i)] 
$$  \pi_j(SU_n) \,\, \simeq \,\, \pi_j(SU_{n+1}) \qquad \text{ for } \quad j 
\leq 2n-1 \, ;  $$
\item[ii)] 
$$  \pi_j(SU_n/SO_n(\R)) \,\, \simeq \,\, \pi_j(SU_{n+1}/SO_{n+1}(\R))  
\qquad \text{ for } \quad j \leq n-1 \, ;  and $$
\item[iii)] 
$$  \pi_j(SU_{2n}/Sp_n) \,\, \simeq \,\, \pi_j(SU_{2(n+1)}/Sp_{n+1})  
\qquad \text{ for } \quad j \leq 4n-2  \, . $$
\end{itemize}
This gives the stable range. 
\end{proof}
\section{Cohomology of the Complement and Link for the Equidimensional 
Case}
\label{S:sec4}
\par 
We next return to the class of special prehomogeneous spaces defined by 
equidimensional representations $\rho : G \to GL(V)$ of a connected linear 
algebraic group.  In this section, we compute the topology of the 
complement and link of the exceptional orbit variety $\cE$.  
\subsection*{Topology of the Complement}
\par
\begin{Thm}
\label{Thm3.2}
Consider a prehomogeneous space defined by an equidimensional 
representation of a connected linear algebraic group $\rho : G \to \GL(V)$ 
with exceptional orbit variety $\cE$ and maximal compact subgroup $K$.  
Then, for a field $\bk$ of characteristic $0$, 
\begin{equation}
\label{Eqn3.1}
 H^*(V\backslash \cE; \bk) \,\, \simeq \,\, H^*(K; \bk) \,\, = \,\,  \gL^* \bk 
\langle s_{1}, s_{2}, \dots , s_{k}\rangle \,\, . 
\end{equation}
where $s_{j}$ are classes of odd degree $q_j$. 
\par
In addition, $\pi_i(V\backslash \cE) \simeq \pi_i(K)$ for $i > 1$; and 
there is a short exact sequence
\begin{equation}
\label{CD3.1}
\begin{CD} 
 {1} @>>> {\pi_1(K)} @>>> {\pi_1(V\backslash \cE)} @>>> {H}  @>>>  {1}
\end{CD}  
\end{equation}
where $H$ is the isotropy subgroup of $G$ at a point 
$v_0 \in V\backslash \cE$.
\end{Thm}
\par  
\begin{proof}
First, for (\ref{Eqn3.1}) we may again apply the Hopf structure theorem 
for a compact connected Lie group $K$ and a field $\bk$ of characteristic 
$0$, $H^*(K; \bk)$, to conclude it is a Hopf algebra which is isomorphic to 
an exterior algebra on classes of odd degree.  
This gives the structure for $H^*(K; \bk)$ in (\ref{Eqn3.1}) 
\par  As $K$ is homotopy equivalent to $G$ we conclude that $H^*(G; \bk)$ 
is also given by the RHS of (\ref{Eqn3.1}).  However, let $v_0 \in \cU$, the 
open orbit of $G$, and let $H$ be the isotropy subgroup of $v_0$, which is 
finite as $\dim_{\C} G = \dim_{\C} V$.  Then, $G/H \simeq \cU = V 
\backslash \cE$.  Thus, we may replace $V \backslash \cE$ by $G/H$. 
Lastly, it is sufficient to show that 
\begin{equation}
\label{Eqn3.2}
  H^*(G/H; \bk) \,\, = \,\,  H^*(G; \bk)  \,\, .  
\end{equation}
This follows in two steps.  First, there is the standard Lemma (see e.g. 
\cite[Chap III, Thm2.4]{Bn}).  
\begin{Lemma}
\label{Lem3.3}
If the finite group $H$ acts freely on a  manifold $M$, then for a field 
$\bk$ of characteristic $0$ or relatively prime to $| H |$, 
$$  H^*(M; \bk)^H \,\, = \,\,  H^*(M/H; \bk)  $$
where $H^*(M; \bk)^H$ denotes the subspace invariant under the induced 
action of $H$ on cohomology
\end{Lemma}
Hence,
\begin{equation}
\label{Eqn3.3}
  H^*(G; \bk)^H \,\, = \,\,  H^*(G/H; \bk)  
\end{equation}
Second, we have the averaging map
\begin{align}
\label{Eqn3.3a}
    avg_H : H^*(G; \bk)  &\to  H^*(G; \bk)^H  \notag  \\
    \left[\tau\right]  &\mapsto \frac{1}{| H |} \sum_{\gs \in H} 
\gs^*(\left[\tau\right]) 
\end{align}
which is an isomorphism.  This follows since as $G$ is connected, for any 
cocycle $\tau$ on $G$, the cohomology classes $\left[\gs^*(\tau)\right] = 
\left[\tau\right]$ for all $\gs \in G$; so $\left[avg_H(\tau)\right] = 
\left[\tau\right]$.  
\par
For the second part, we use the long exact sequence in homotopy for the 
fibration $H \hookrightarrow G \to G/H$ (just as in \cite[\S 3]{DP}) to 
obtain the desired isomorphisms and the exact sequence.  
\end{proof}
\par
We state an immediate consequence for the case of an equidimensional 
representation $\rho : G \to \GL(V)$ of a connected solvable linear 
algebraic group with an open orbit.  As $G$ has a maximal compact 
subgroup $T^k$, where $k = \rank (G)$, we obtain from 
Theorem~\ref{Thm3.2}.
\begin{Corollary}
\label{Cor3.5}
Suppose $\rho : G \to \GL(V)$ is an equidimensional representation of a 
connected solvable linear algebraic group which defines a prehomogeneous 
space with exceptional orbit variety $\cE$.  If $\rank (G) = k$, then $V 
\backslash \cE$ is a $K(\pi, 1)$ with $\pi$ a finite extension of $\Z^k$ by 
the finite isotropy group $H$ of a point in $\cU$; and
$$ H^*(V \backslash \cE; \bk) \,\, = \,\, \gL^* \bk \langle s_{1}, s_{2}, 
\dots , s_{k}\rangle  \,\, . $$ 
where each $s_j$ is of degree one.
\end{Corollary}
\par
\begin{Remark}
\label{Rem3.1}
Since an equidimensional representation of a connected solvable linear 
algebraic group with open orbit $\cU$ is a (possibly nonreduced) block 
representation, the first part of Corollary \ref{Cor3.5} was given in 
\cite{DP}.  The second part extends for cohomology the results in 
\cite{DP} which for the special cases of representations corresponding to 
(modified) Cholesky-type factorizations furthermore showed that the 
complement and Milnor fibers were homotopy tori.  
\end{Remark}
\par
\subsection*{Topology of the Link}
\par 
We can immediately deduce 
the cohomology of the link of the exceptional orbit variety $\cE$.
For an exterior algebra $A = \gL^* \bk \langle s_{1}, s_{2}, \dots , 
s_{k}\rangle$, with generators $s_i$ of odd degrees $q_i$,we let 
$$ \widetilde{\gL^* \bk} \langle s_{1}, s_{2}, \dots , s_{k}\rangle\left[ 
r\right] $$
denote the algebra obtained from $A$ by removing the top degree term, and 
shifting degrees upward by degree $r$.  Then, there is the following result 
for the link $L(\cE)$ of $\cE$.
\begin{Thm}
\label{Thm3.6}
Let $\rho : G \to \GL(V)$ be an equidimensional representation of a linear 
algebraic group defining a prehomogeneous space with exceptional orbit 
variety $\cE$, and a maximal compact subgroup $K$.  Then, for a field 
$\bk$ of characteristic $0$, there is an isomorphism of graded vector 
spaces
$$ \widetilde{H}_*(L(\cE); \bk) \,\, \simeq \,\, \widetilde{H}^*(L(\cE); 
\bk) \,\, \simeq \,\,  \widetilde{\gL^* \bk} \langle s_{1}, s_{2}, \dots , 
s_{k}\rangle\left[ 2N - 2 -\dim_{\R}K\right]  $$
where $N = \dim_{\C} V = \dim_{\C} G$ and $k = \rank(G)$.
\end{Thm}
\par
\begin{proof}
This result follows from Proposition \ref{Prop1.8} using Theorem 
\ref{Thm3.2}.
\end{proof}
\begin{Corollary}
\label{Cor3.7}
In the situation of Theorem \ref{Thm3.6},  
\begin{itemize}
\item[i)] If $G$ is solvable and $\not\simeq (\C^*)^N$, then the degree 
shift $> N-2$; hence products in the reduced cohomology  
$\widetilde{H}^*(L(\cD); \bk)$ are $0$.  
\item[ii)]  When $G$ is reductive, the degree shift equals $N - 2$, so  
products in the reduced cohomology $\widetilde{H}^*(L(\cD); \C)$ are $0$ 
except possibly for the single product $e_1^* \cup 1^*$, where the 
$e_1^*$ and $1^*$ are the images in degrees $N-1$, resp. $N-2$, of $e_1, 1 
\in H^*(M \backslash \cD; \bk)$.  
\end{itemize}
\end{Corollary}
\begin{proof}[Proof of Corollary \ref{Cor3.7}]
If $G$ is solvable and $\not\simeq (\C^*)^N$ then $\dim_{\R} K < 
\dim_{\C} G = N$; hence, $2N - 2 - \dim_{\R} K > N-2$.  Thus the product of 
two classes in $\widetilde{H}^*(L(\cD); \bk)$ have degree $\geq 2N -2$ 
and hence is $0$.  \par
For the equidimensional case when $G$ is reductive, $\dim_{\C} G = 
\dim_{\R} K$ and $H = L$ is finite, so $N = \dim_{\C} G = \dim_{\R} K/L$.  
Hence, the degree shift is simply $N - 2$.  \par
Then, the argument follows as for the analogous property in Theorem 
\ref{Thm2.2}.
\end{proof}
\par
As a second corollary we consider the number of irreducible components 
of $\cE$.  Each component $W_i$ such will contribute a generator from the 
fundamental class 
$\left[ W_i \cap S^{2N-1}\right]$ to $H_{2N-3}(L(\cE); \bk)$.  Applying 
Theorem \ref{Thm3.6} we obtain as a corollary
\begin{Corollary}
\label{Cor3.8}
For a representation $\rho : G \to \GL(V)$ as in Theorem \ref{Thm3.6}, the 
number of irreducible components of $\cE$ equals $\dim H_1(K; \bk)$.  
In particular, when $G$ is solvable, the number equals $\rank(G)$. 
\end{Corollary}
\begin{proof}
The generators of
$$  H_{2N-3}(L(\cE); \bk) \,\, \simeq \,\, H^{2N-3}(L(\cE); \bk) \, .$$
are given by the fundamental classes $\left[ W_i \cap S^{2N-1}\right]$ for 
the components $W_i$ of $\cE$.  Thus, by Theorem \ref{Thm3.6}, the 
number of components equals $\dim_{\bk} H_1(K; \bk)$.  \par
If $G$ is solvable of rank $k$, then $K = T^k$ so $\dim_{\bk} H_1(T^k; \bk) 
= k$
\end{proof}
\begin{Example}
\label{Ex3.9}
We compare the topology of the links for the exceptional orbit 
hypersurfaces in $M_{2, 2}$ defined for $2 \times 2$ matrices 
$\begin{pmatrix} x & y \\ z & w \end{pmatrix}$:  $\cD_2$, the 
determinantal hypersurface of singular matrices (defined by $x w-y z$), 
for the representation of left multiplication by $\GL_2(\C)$; the linear 
free* divisor $\cE_2^{\prime}$ (defined by $x (x w-y z)$) for the 
representation for Cholesky factorization using the solvable group $B_2 
\times N_2^T$ (with $B_2$ the Borel subgroup of lower triangular 
matrices and $N_2^T$ the upper triangular unipotent matrices); and the 
linear free divisor $\cE_2$ (defined by $x y (x w-y z)$) for the 
representation for modified Cholesky factorization using the solvable 
group $B_2 \times C_2^T$ (with $B_2$ the Borel subgroup of lower 
triangular matrices and $C_2^T \simeq \C^*$ the invertible diagonal 
matrices with top entry $1$)  (also see \cite{DP2} for the second and 
third).  Table \ref{Comp.excep.hyprs} gives the maximal compact 
subgroups $K$, compact model for the Milnor fiber, and the reduced 
cohomology of the links in the nonvanishing dimensions.  These exhibit the 
increased complexity of the cohomology of the link, and changes in the 
topology of the Milnor fiber resulting from successively adding two 
hyperplanes to obtain the linear free divisor $\cE_2$.  
\end{Example}

\begin{table}[h]
\begin{tabular}{|l|c|c|c|c|c|c|l|}
\hline
$\cE$  &$h$  & $K$  &  $F$  & $\widetilde{H}^J(L(\cE),\bk)$  &  &  &   \\
  &  &  &  model  &  $j = 2$  & $3$ & $4$ &$5$   \\
\hline
$\cD_2$ & $x w - y z$ & $U_2$ & $SU_2 \simeq S^3$  &  $\bk$ & $\bk$ & 0  
& $\bk$\\
\hline
$\cE_2^{\prime}$ & $x (x w - y z)$ & $T^2$ & $S^1$  &  $0$ & $0$ & $\bk$  
& $\bk^2$  \\
\hline
$\cE_2$  & $x y (x w - y z)$ & $T^3$ & $T^2$  &  $0$ &  $\bk$ & $\bk^3$  &  
$\bk^3$\\
\hline
\end{tabular}
\caption{\small Three exceptional orbit hypersurfaces arising from 
equidimensional representations in Example \ref{Ex3.9}, with defining 
equation $h = 0$, together with the maximal compact subgroup $K$ of $G$, 
the compact model for the global Milnor fiber $F$, and the nonzero reduced 
cohomology groups $\widetilde{H}^J(L(\cE),\bk)$.  }
\label{Comp.excep.hyprs}
\end{table}

\section{Topology of the Milnor Fiber for Equidimensional 
Representations}
\label{S:sec5}  
\par
We next apply the preceding results on the cohomology of the complement 
together with properties of the the Wang sequence in \S \ref{S:sec1} and 
the results for equidimensional representations in \S \ref{S:sec2} to 
compute the topology of the Milnor fiber of $\cE$.  We use the notation of 
\S\S \ref{S:sec1} and \ref{S:sec2}. \par
\begin{Thm}
\label{Thm4.2}
Consider a prehomogeneous space defined by an equidimensional 
representation $\rho : G \to \GL(V)$ of a connected linear algebraic group 
$G$ of rank $k$ with maximal compact subgroup $K$ and exceptional orbit 
variety $\cE$.  Then, the global Milnor fibration of $\cE$ is 
cohomologically trivial.  Furthermore, for a field $\bk$ of characteristic 
$0$, 
$$  H^*(F; \bk) \,\, \simeq  \,\, \gL^* \bk <e_2, \dots , e_k>  \, .$$
Here $e_j = \iti^*(s_j)$, where $\iti : F \hookrightarrow V \backslash 
\cE$ is the inclusion and  
$$ H^*(V \backslash \cE; \bk) \,\, \simeq  \,\, H^*(K; \bk) \,\, \simeq  
\,\,\gL^* \bk <s_1, s_2, \dots , s_k> $$ 
with $\deg s_1 = 1$.   \par
Moreover, the homotopy groups of $F$ are given by $\pi_j(F) \simeq 
\pi_j(G)$ for $j \geq 2$; and there is the exact sequence 
\begin{equation}
\label{CD5.1}
\begin{CD} 
{1} @>>>  {\pi_1(G^{\prime}_0)} @>>>  {\pi_1(F)} @>>> {H} @>>> {1}  
\end{CD}
\end{equation}
where $H$ is the isotropy group of $G^{\prime}_0$ for a point in $F$ and 
$\pi_1(G^{\prime}_0)$ is in the exact sequence
\begin{equation}
\label{CD5.2}
\begin{CD} 
{1} @>>>  {\pi_1(G^{\prime}_0)} @>>>  {\pi_1(G)} @>>> {\Z} @>>> {1}  
\end{CD}
\end{equation}
\end{Thm}
\begin{proof} 
We can apply apply Proposition \ref{Prop1.8} and Theorem \ref{Thm3.2}, to 
compute for the global Milnor fibration, the cohomology 
\begin{equation}
\label{Eqn5.5}
H^*(E; \bk) \,\, \simeq \,\, H^*(V\backslash \cE; \bk) \,\, \simeq \,\, 
H^*(K; \bk)\, ,
\end{equation}
 where 
\begin{equation}
\label{Eqn5.6}
H^*(K; \bk) \,\, \simeq \,\, \gL^*\bk\langle s_1, \dots , s_k\rangle   
\end{equation}
with $k = \rank K$.  \par
Second, by Lemma \ref{Lem1b.3}, $G^{\prime}_0$ is a finite covering 
space of $F$, so by Lemma \ref{Lem3.3} and then Lemma \ref{Lem1b.2}, 
\begin{equation}
\label{Eqn5.7}
H^*(F; \bk) \,\, \simeq \,\,  H^*(G^{\prime}_0; \bk) \,\, \simeq \,\,  
H^*(K^{\prime}_0; \bk)  
\end{equation}
Finally by Lemma \ref{Lem1b.4}, since $K^{\prime}_0 \hookrightarrow K 
\to  S^1$ is cohomologically trivial, we have by Proposition \ref{Prop1.5} 
that $\dim_{\bk} H^*(K; \bk) = 2 \,\dim_{\bk} H^*(K^{\prime}_0; \bk)$.
Combining the preceding we obtain
\begin{equation}
\label{Eqn5.8}
 \dim_{\bk} H^*(E; \bk)\,\, = \,\, 2\, \dim_{\bk} H^*(F; \bk) \, . 
\end{equation}
Thus, by Proposition \ref{Prop1.5} the global Milnor fibration is 
cohomologically trivial.  
Also, by the discussion for (\ref{CD1.4}) in \S \ref{S:sec1}, there is a 
nontrivial class $s_1 \in H^1(E; \bk)$, and we may choose  $s_1$ to be one 
of the generators of the exterior algebra.  Thus, combining (\ref{Eqn5.5}) 
and (\ref{Eqn5.6}) we have an isomorphism of algebras 
\begin{equation}
\label{Eqn5.12}
  H^*(E; \bk)  \,\, \simeq \,\, \gL^* \bk \langle s_2, \dots , s_k\rangle \,\, 
\oplus \,\, \left( \langle s_1 \rangle \cdot \gL^* \bk \langle s_2, \dots , 
s_k\rangle\right) \, . 
\end{equation} 
where the product in the second summand is cup prouct.
\par
As the cohomology of the global Milnor fibration is cohomologically 
trivial, the second summand on the RHS of (\ref{Eqn5.12}) maps by $\iti^* : 
H^*(E; \bk) \to H^*(F; \bk)$ to be $0$.  It follows by consideration of 
dimensions and Proposition \ref{Prop1.5} that the first summand maps by 
$\iti^*$ isomorphically to $H^*(F; \C)$. This gives the desired conclusion. 
\par
For the homotopy groups $\pi_j(F)$, we use the finite covering space 
$G^{\prime}_0 \to F$ with covering group $H^{\prime}$.  The exact 
sequence (\ref{CD5.1}) follows from the basic relation between 
fundamental groups for a regular covering, and the higher homotopy groups 
follow from e.g. the long exact homotopy sequence for a fibration.  Lastly 
the exact sequence (\ref{CD5.2}) follows from the long exact homotopy 
sequence of the fibration $G^{\prime}_0 \hookrightarrow G \to \C^*$ since 
both $G^{\prime}_0$ and $G$ are connected.
\end{proof}
\par
\subsection*{Linear Free and Free* Divisors for Solvable Linear Algebraic 
Groups} 
\par
We consider the case of an equidimensional representation $\rho : G \to 
\GL(V)$ defining a prehomogeneous space when $G$ is solvable, with 
exceptional orbit variety $\cE$.  As $G$ has a maximal compact subgroup 
$T^k$, where $k = \rank (G)$, we already know by Theorems~\ref{Thm3.6} 
and \ref{Thm4.2} that the cohomology of both the complement 
$V\backslash \cE$ and Milnor fiber $F$ are an exterior algebras on $k$, 
resp. $k-1$ generators, all of which are of degree $1$.   \par 
As in Theorem 4.1 of \cite{DP}, we can explicitly construct the generators 
from the basic relative invariants $p_i$ of the representation. $\rho$.  By 
Corollary \ref{Cor3.8}, there are $k = \rank G$ irreducible components 
$W_i$ of $\cE$.  By \cite{SK}, the homogeneous defining equations $p_i$ 
for $W_i$ are {\em basic relative invariants} and are independent.  We let 
$\gw_i = \frac{dp_i}{p_i} = p_i^*(\frac{dz}{z})$, which provide $k$ closed 
one-forms on $\cU = V \backslash \cE$.  The reduced homogeneous 
defining equation for $\cE$ is given by $f = \prod_{i = 1}^{k} p_i$.  Then, 
$\tilde \gw = \frac{df}{f}$ defines a cohomology class in $H^1(V 
\backslash \cE; \C)$ whose restriction $\tilde \gw | F = 0$ as $df | F = 0$  
(since $f \equiv 1$ on $F$).  \par
We deduce an extension of the result of Damon-Pike \cite{DP}.
\begin{Thm}
\label{Thm4.1}
Let $\rho : G \to \GL(V)$ be an equidimensional representation of a 
solvable linear algebraic group $G$ of rank $k$, defining a prehomogeneous 
space with exceptional orbit variety $\cE$.  Then, \par
\begin{itemize}
\item[i)] $H^1(V \backslash \cE, \C)$ is the exterior algebra on  the set of 
generators $\gw_i$ for $i = 1, \dots k$. 
\item[ii)] $H^1(F, \C)$ is generated by the $\{ \gw_i, i = 1, \dots , k\}$ 
with a single relation $\sum_{i = 1}^{k}  \gw_i = 0$.  Hence,  $H^*(F, \C)$ 
is the exterior algebra on any subset of $k -1$ of the $\gw_i$.
\end{itemize}
\end{Thm}
\begin{proof}
\par
We know by Theorem \ref{Thm3.2} that $H^*(V \backslash \cE; \C)$ is an 
exterior algebra on $k$ generators of degree $1$.  Then, we can apply same 
the methods in \cite[\S 4]{DP} to show that the $k$ closed $1$-forms 
$\gw_i$ defined from the basic relative invariants $p_i$ are a set of 
generators for $H^1(V \backslash \cE; \C)$; and hence, are exterior algebra 
generators for $H^*(V \backslash \cE; \C)$.  By Theorem \ref{Thm4.2}, the 
complex cohomology of the global Milnor fiber is an exterior algebra on 
the pull-backs of $k-1$ of the generators $\{ \gw_i\}$ for $H^1(V 
\backslash \cE; \C)$.   Since the pullback of $\sum_{i = 1}^{k} \gw_i = 
\tilde \gw = 0$, and $k-1$ of the generators will suffice.  \par
\end{proof}

\subsection*{Equidimensional Representations of Reductive Linear 
Algebraic Groups} 
\par 
In the case where $G$ is a connected reductive linear algebraic group with 
maximal compact subgroup $K$, and $\rho : G \to \GL(V)$ defines a 
prehomogeneous space, then the results apply even if the exceptional orbit 
variety $\cE$ is not a linear free divisor.  However, we give one 
consequence for an important class of such representations for quivers of 
finite type considered by Buchweitz-Mond which do give linear free 
divisors.  \par
\subsubsection*{Quivers of Finite Representation Type}
\par
A quiver is a connected finite directed graph $\gG$ with edges $e(\gG) = 
\{\ell_j\}$, vertices $v(\gG) = \{v_i\}$, where we denote the initial 
vertex for $\ell_j$ by $\iti(\ell_j)$ and end point by $\ite(\ell_j)$.  To 
define a {\em representation of the quiver} $\gG$, we associate to each 
vertex $v_i$ a finite dimensional complex vector space $V_i$ of 
dimension $d_i$ and for each edge $\ell_j$ a linear transformation 
$\varphi_j : V_{\iti(\ell_j)} \to V_{\ite(\ell_j)}$.  Then, the tuple $\{ 
\varphi_j\}_{\ell_j \in e(\gG)}$ of linear transformations $\varphi_j : 
V_{\iti(\ell_j)} \to V_{\ite(\ell_j)}$ is a quiver representation.  With an 
ordering on the vertices $\{v_i \in v(\gG)\}$, we let $\bd = (d_i)$ denote 
the {\em dimension vector} for the quiver representation.  \par 
Together a set of such transformations $\{\varphi_j\}$ forms a quiver 
representation space $V \simeq \prod_{\ell_j \in \cL} 
\Hom(V_{\iti(\ell_j)}, V_{\ite(\ell_j)})$.  The group $\tilde G = 
\prod_{v_i \in v(\gG)} \GL(V_i)$ acts on $V$ by 
$$ \{ \psi_i\} \cdot \{\varphi_j\} \,\, = \,\, \{\psi_{\ite(\ell_j)}\circ 
\varphi_j\circ \psi_{\iti(\ell_j)}^{-1}\}  \qquad { for } \quad \{ \psi_i\} 
\in \tilde G,\,\,  \{\varphi_j\} \in V  $$
The group $\C^*$ embeds in each $\GL(V_i)$ as the standard 
$\C^*$-action.  The diagonal embedding of $\C^*$ in $\tilde G$ defines a 
subgroup of $\tilde G$ which acts trivially on $V$.  We let $G = \tilde 
G/\C^*$.  Then, the quiver is said to be of {\em finite representation type} 
if $G$ has only a finite number of isomorphism classes of {\em 
indecomposable quiver representations}.  \par 
 The classification of quivers of finite representation type was done by 
Gabriel \cite{G, G2} and they correspond to the Dynkin diagrams of type 
$A$, $D$, or $E$ with the indecomposable quiver representations for 
appropriate dimension vectors $\bd$ which are the positive Schur roots 
corresponding to the Dynkin diagram.  The quiver arrows can go in either 
direction.  Then, it is a fact that for these dimension vectors the 
representation of $G$ on $V$ is an equidimensional representation with an 
open orbit, which hence defines a prehomogeneous space whose 
exceptional orbit variety, denoted $\cD_{(\gG, \bd)}$, is called the {\em 
discriminant of the quiver}.  
Buchweitz-Mond prove in \cite{BM} that for quivers of finite 
representation type the discriminant $\cD_{(\gG, \bd)}$ is a linear free 
divisor.  \par
As a result of Theorems~\ref{Thm4.2} and \ref{Thm3.6}, we can compute 
the cohomology in characteristic $0$ of both the link and Milnor fiber of 
the quiver discriminant.  

To do so we need a simple lemma. 
\par 
Let $G$ be a connected linear algebraic group with $\itj : \C^* 
\hookrightarrow G$ a subgroup, such that for $S^1 \subset \C^*$, 
$\itj^{\,\prime} : S^1 \subset K$ for $K$ a maximal compact subgroup of 
$G$, of rank $k$.  
By Hopf\rq s structure theorem, 
\begin{equation}
\label{Eqn4.4}
H^*(K; \bk) \,\, \simeq  \,\, \gL^* \bk <s_1, \dots , s_k>  \, .
\end{equation}
with each $s_i$ of odd degree $q_i$ and $s_1$ of degree $1$.  
\begin{Lemma}
\label{Lem4.3}
In the preceding situation, suppose that $\itj^{\,\prime\, *}(s_1)$ 
generates  $H^1(S^1; \bk)$.  Then, 
\begin{equation}
\label{Eqn4.5}
  H^*(G/\C^*; \bk) \,\, \simeq  \,\, \gL^* \bk <s_2, \dots , s_k> \, .
\end{equation}
\end{Lemma}
\begin{proof}
By the same argument given in Lemma \ref{Lem1.2}, $K/S^1 
\hookrightarrow G/\C^*$ is a homotopy equivalence.  Hence it is sufficient 
to prove the result for $K/S^1$.  Then, $\{ 1, s_1\}$ restrict to a basis for 
$H^*(S^1; \bk)$.  Hence, the Leray-Hirsch theorem applied to the fibration 
$S^1 \hookrightarrow K \overset{p}\to K/S^1$ yields that $H^*(K; \bk)$ is 
a free $H^*(K/S^1; \bk)$-module on $\{ 1, s_1\}$, where the module 
structure is via $p^*$ and $p^*$ is injective.  
Also by (\ref{Eqn4.4}), $H^*(K; \bk)$ is a free $\gL^* \bk <s_2, \dots , 
s_k>$-module on $\{ 1, s_1\}$.  
Thus,
$$ H^*(K; \bk)/ (s_1\smile H^*(K; \bk)) \,\, \simeq  \,\, \gL^* \bk <s_2, 
\dots , s_k> \, ;$$
and also by the Leray-Hirsch theorem, 
$$ H^*(K; \bk)/ (s_1\smile H^*(K; \bk)) \,\, \simeq  \,\, p^*H^*(K/S^1; 
\bk)\cdot \{1\} \, .$$
Thus, if we compose $p^*$ with projection onto the quotient by cup 
product with $s_1$.  We obtain an isomorphism 
$$  H^*(K/S^1; \bk) \,\, \simeq  \,\, \gL^* \bk <s_2, \dots , s_k> \, . $$
\end{proof}
\par 
Now in the case of quivers for a graph $\gG$ and quiver representation 
space $V$ with dimension vector $\bd$, we have the subgroup $\C^* 
\hookrightarrow \tilde G = \prod \GL(V_i)$.  The maximal compact 
subgroup of $\GL(V_i) \simeq \GL_{d_i}(\C)$ is the unitary group  
$U_{d_i}$, which has cohomology
$$ H^*(U_{d_i}; \bk) \,\, \simeq  \,\, \gL^* \bk \langle s^{(i)}_1, \dots , 
s^{(i)}_{d_i}\rangle \, .$$
with $s^{(i)}_j$ of degree $2j -1$.  Then, $\tilde K = \prod_{v_i \in v(\gG)} 
\in  U_{d_i}$ is the maximal compact subgroup of $\tilde G$ and its 
cohomology is given by  
\begin{equation}
\label{Eqn4.6}
  H^*(\tilde K; \bk)  \,\, \simeq \,\, \otimes_{v_i \in v(\gG)} \gL^* \bk 
\langle s^{(i)}_1, \dots , s^{(i)}_{d_i}\rangle  
\end{equation}
which is again an exterior algebra.  We denote the RHS of (\ref{Eqn4.6}) by 
$\gL^*(\gG, \bd)$.  We let $K = \tilde K/S^1$ which is the maximal 
compact subgroup of $G$.  \par
To compute the cohomology of $K$ we use Lemma \ref{Lem4.3}. 
First, the degree $1$ generator $s^{(i)}_1$ arises at the pull-back of the 
generator of $H^1(S^1; \bk)$ via $\det : U_{d_i} \to S^1$.  The composition 
of $S^1 \hookrightarrow U_{d_i}$ with the determinant map sends $z \to 
z^{d_i}$, which is a covering map but still induces an isomorphism on 
cohomology with coefficients in $\bk$.  Thus, the pull-back of $s^{(i)}_1$ 
to $H^1(S^1; \bk)$ via the inclusion is a generator.  \par 
Hence, the  pull-back of any degree $1$ generator $\sum_{v_i \in v(\gG)} 
a_i s^{(i)}_1$ not belonging to a certain codimension one subspace of  
$H^*(U_{d_i}; \bk)$ will generate $H^1(S^1; \bk)$.  We choose such a 
generator $s_1$, which can be chosen to be one of the generators of the 
exterior algebra.  Then, we have by Lemma \ref{Lem4.3}
\begin{equation}
\label{Eqn4.7}
 H^*(K; \bk)  \,\, \simeq \,\, \gL^*(\gG, \bd)/ (s_1 \cdot \gL^*(\gG, \bd)) 
\end{equation}
which is an exterior algebra obtained by removing a generator of degree 
$1$ from (\ref{Eqn4.6}).  \par
Second, by the discussion for (\ref{CD1.4}) in \S \ref{S:sec1} and Lemma 
\ref{Lem1b.4} in \S \ref{S:sec2}, there is a nonzero class $s_2 \in H^1(K; 
\bk)$ obtained as the pull-back of the generator of $H^1(S^1; \bk)$ via the 
composition of the map $\chi^{\prime} : K \to S^1$ with the projection 
$\tilde K \to \tilde K/S^1 = K$.  This will be different from $s_1$.  Then, 
we have the following structure theorem for the cohomology of the Milnor 
fiber and link of the quiver discriminant.
\begin{Thm}
\label{Thm4.8}
For the quiver representation space $V$ for a quiver $\gG$ of finite type 
with dimension vector $\bd$, having an open orbit of indecomposable 
quiver representations, let $F_{(\gG, \bd)}$ denote the Milnor fiber of the 
discriminant $\cD_{(\gG, \bd)}$, and $L(\cD_{(\gG, \bd)})$ the link.  Then,
\begin{equation}
\label{Eqn4.8a}
 H^*(F_{(\gG, \bd)}; \bk)  \,\, \simeq \,\, \gL^*(\gG, \bd)/ (s_1, s_2) \cdot 
\gL^*(\gG, \bd) 
\end{equation}
which is the exterior algebra on the generators of (\ref{Eqn4.6}) but with 
two degree $1$ generators removed.  Also, 
\begin{equation}
\label{Eqn4.8b} 
\widetilde{H}^*(L(\cD_{(\gG, \bd)}); \bk)  \,\, \simeq \,\, 
\widetilde{\gL^*(\gG, \bd)}/ (s_1 \cdot \widetilde{\gL^*(\gG, \bd)})
[\dim_{\C}V - 2] 
\end{equation}
which is the exterior algebra on the generators of (\ref{Eqn4.6}) with one 
degree $1$ generator removed, then truncated in the top degree, and then 
shifted by degree $\dim_{\C}V - 2$.
\end{Thm}
\begin{Remark}
\label{Rem4.10}
For an odd integer $k > 1$, the number of generators in degree $k$ of the 
exterior algebras in (\ref{Eqn4.8a}) and (\ref{Eqn4.8b}) is given by $| \{ 
v_j \in v(\gG) : d_j \geq \frac{k+1}{2}\} |$.  While the number of degree 
$1$ generators is either $2$, resp. $1$, less than $| v(\gG) |$.
\end{Remark}
\begin{proof}
This follows from from the preceding discussion, Lemma \ref{Lem4.3}, 
together with Theorems~\ref{Thm4.2} and \ref{Thm3.6}.
\end{proof}
\par
\begin{Example}
\label{Ex4.11}
We consider the quiver representation corresponding to $D_4$ which has 
$4$ vertices with a central one $v_1$ connected to the other three by 
edges, with the direction toward the central vertex.  The dimension vector 
$\bd$ has dimensions $\dim_{\C} V_1= 2$, and for the other vertices 
$\dim_{\C} V_i= 1$.  By Mond and Buchweitz, we can view the linear 
transformations as determined by vectors $(x_i, y_i) \in \C^2$, for $i = 1, 
2, 3$.  The quiver discriminant consists of the triple of vectors for which 
at least two of them lie in a common line.  This is defined by the equation 
$(x_1y_2 - x_2y_1)(x_1y_3 - x_3y_1)(x_2y_3 - x_3y_2) = 0$ in $\C^6$.  
The group $G = \left( \GL(\C^2) \times (\C^*)^3\right)/ \C^*$ has maximal 
compact subgroup $K = (U_2 \times T^3)/S^1$.  Then, by Theorem 
\ref{Thm4.8} we have the cohomology of the Milnor fiber 
\begin{equation}
\label{Eqn4.10a}
 H^*(F_{(D_4, \bd)}; \bk)  \,\, \simeq \,\, \gL^*\bk \langle s_1, s_2, 
s_3\rangle 
\end{equation}   
and 
\begin{equation}
\label{Eqn4.10b}
 \widetilde{H}^*(L(\cD_{(D_4, \bd)}); \bk)  \,\, \simeq \,\, 
\widetilde{\gL^*\bk} \langle s_1, s_2, s_3, s_4\rangle[4] 
\end{equation}
where we let $s_1$ have degree $3$ and the other $s_i$ have degree $1$. 
Thus, $H^j(F_{D_4}; \bk)$ is nonzero in degrees $0 \leq j \leq 5$ and has  
dimensions $1, 2, 1, 1, 2, 1$; while $\widetilde{H}^j(L(\cD_{(D_4, \bd)}); 
\bk)$ is nonzero in degrees $4 \leq j \leq 9$ and has dimensions $1, 3, 3, 
2, 3, 3$.  The top degree $9$ of dimension $3$ corresponds to the $3$ 
irreducible components of $\cD_{(D_4, \bd)}$.  
\end{Example}
\begin{Remark}
\label{Rem5.5}
Mond and Buchweitz have examined the effect of reversing various quiver 
arrows in a quiver representation of finite type.  David Mond has indicated 
they have found that for certain quiver representations, the discriminant 
changes when the directions of certain quiver arrows are changed.  
However, it follows from the results obtained here that the rational 
cohomology of the Milnor fiber,  the complement, and the cohomology of 
the link (as a graded vector space) will not change.  A natural question is 
which topological invariants will detect this change.  
\end{Remark}
\section{Topology of Formal Sums of Exceptional Orbit Hypersurfaces}  
\label{S:sec6}
\par
We conclude by combining the preceding results with a result of Mutsuo 
Oka \cite{Ok}, together with the results from Siersma et al to exhibit a 
large collection of highly nonisolated hypersurface singularities whose 
Milnor fibers are either joins of compact manifolds or bouquets of 
suspensions of such a join, and whose topology we can explicitly compute. 
\par 

\subsection*{Topology of Formal Sums of Exceptional Orbit 
Hypersurfaces} 
We use the notion of a {\em formal linear combination of hypersurface 
singularities}.  We consider $f_i : \C^{n_i}, 0 \to \C, 0$ for $i = 1, \dots , 
r$, defining hypersurfaces $X_i, 0 \subset \C^{n_i}, 0$.  We regard the 
$\C^{n_i}$ as distinct spaces and let $\pi_i : \prod_{i = 1}^{r} \C^{n_i} \to 
\C^{n_i}$ denote projection on the $i$-th factor.  Then, for $a_i \in \C^*$ 
we let $f = \sum_{i = 1}^{r} a_i f_i\circ \pi_i$.  Then, $f$ defines a 
hypersurface in $\prod_{i = 1}^{r} \C^{n_i}$ which we will refer to as the 
{\em formal linear combination of the hypersurfaces defined by the $f_i$}. 
We will denote it by $f = \oplus_{i = 1}^{r} a_i f_i$.  
\par  
We are interested in the special case where the $f_i$ define exceptional 
orbit hypersurfaces $\cE_i$ of the types we considered earlier.   Then, we 
will first combine the earlier results we obtained for them with the 
following result of Mutsuo Oka \cite[Thm 1]{Ok} which considerably 
extends a classical result of Thom-S\'ebastiani.  \par
\begin{Thm}[Oka]
\label{Thm6.1}
  Let $g: \C^n, 0 \to \C, 0$ and $h: \C^m, 0 \to \C, 0$ be weighted 
homogeneous germs with Milnor fibers $Y$, resp. $Z$, then $f = g \oplus h : 
\C^{n+m}, 0 \to \C, 0$ has Milnor fiber $X$ homotopy equivalent to the join 
$Y * Z$.  Moreover, the monodromy of $f$ is the join of the monodromies of 
$g$ and $h$.
\end{Thm}
To use this result we recall the consequences for cohomology with 
coefficients in a field $\bk$ of characteristic $0$ (see e.g. \cite{Ok} or 
\cite[Chap. 5]{CF}). 
\begin{equation}
\label{eqn6.2}
\tilde{H}^{\ell}(Y*Z; \bk) \,\, \simeq \,\, \oplus_{i = 0}^{\ell - 1}\,\,  
\left( \tilde{H}^{i}(Y; \bk) \otimes \tilde{H}^{\ell-i -1}(Z; \bk)\right)
\end{equation}
Hence, we have the following isomorphism of cohomology viewed as 
graded vector spaces 
\begin{equation}
\label{eqn6.3}
\tilde{H}^{*}(Y*Z; \bk) \,\, \simeq \,\,  (\tilde{H}^{*}(Y; \bk) \otimes 
\tilde{H}^{*}(Z; \bk))[1]
\end{equation}
where as earlier \lq\lq$[1]$\rq\rq will denote shift increasing degrees by 
$1$.  
Then, if the monodromy on $H^*(Y; \bk)$ is denoted by $\gs_g$ and that on 
$H^*(Z; \bk)$, by $\gs_h$, then the join of the monodromies, denoted 
$\gs_g * \gs_h$, which by Theorem \ref{Thm6.1} gives the monodromy 
$\gs_f$, is given by the tensor product $\gs_g \otimes \gs_h$ on each 
summand.    \par
We may combine Oka\rq s theorem with our earlier results.  For $i = 1, 
\dots , r$, let $f_i: \C^{n_i}, 0 \to \C, 0$ define an exceptional orbit 
hypersurface $\cE_i, 0$ with global Milnor fiber $F_i$ with model 
compact submanifold $M_i$, global Milnor fibration $p_i : E_i \to S^1$, and 
monodromy $\gs_i$ on $\tilde{H}^*(F_i; \bk)$.  \par
\begin{Thm}
\label{Thm6.4}
Let $f : \C^n, 0 \to \C, 0$ be given by the formal linear combination  
$\oplus_{i = 1}^{r} a_i f_i$ with $a_i \in \C^*$, and $f_i$ as above.  It 
defines a hypersurface $\cE$, with global Milnor fiber $F$, global Milnor 
fibration $p : E \to S^1$, and monodromy $\gs$.  Then 
\begin{itemize}
\item[i)] $F$ is homotopy equivalent to the join of compact manifolds 
$M_1 * M_2 * \cdots * M_r$;
\item[ii)] hence, 
\begin{equation}
\label{eqn6.4b}
\tilde{H}^*(F; \bk) \,\, \simeq \,\, \left( \otimes_{i = 1}^{r} \,\, 
\tilde{H}^*(M_i; \bk) \right)[r-1]. 
\end{equation}
\item[iii)] Also, the monodromy $\gs = \gs_1 * \gs_2 * \cdots * \gs_r$; 
so \item[iv)] if the Milnor fibration of each $f_i$ is cohomologically 
trivial, then the Milnor fibration of $f$ is cohomologically trivial.
\item[v)]  In the case of iv), 
\begin{equation}
\label{eqn6.5}
H^{*}(\C^n \backslash \cE; \bk) \,\, \simeq \,\,  H^{*}(E; \bk) \,\, \simeq 
\,\,  \gL^*\bk\{s_1\} \otimes \left(\bk\{ 1\} \oplus (\otimes_{i = 1}^{r} 
\,\, \tilde{H}^*(M_i; \bk)) [r-1]\right)
\end{equation}
where $\bk\{ 1\}$ has degree $0$ and $\deg (s_1) = 1$. 
\end{itemize}
\end{Thm}
\begin{proof}
For i) ii) and and iii), we use induction on the number $r$ of terms in the 
formal linear combination $f = \oplus_{i = 1}^{r} a_i f_i$. If the result is 
true when $m < r$ then $f = g \oplus h$ where $g = \oplus_{i = 1}^{r-1} a_i 
f_i$ and $h = a_r f_r$.  As each $f_i$ is homogeneous on different 
coordinates, both $g$ and $h$ are weighted homogeneous for a common set 
of weights on $\C^n$, so Oka\rq s theorem applies and the Milnor fiber $X = 
F$ of $f$ is homotopy equivalent to $Y*Z$ with $Y$ and $Z$ denoting the 
global Milnor fibers of $g$, resp. $h$ (and the Milnor fiber of $h$ is 
diffeomorphic to that of $f_r$, i.e. $F_r$).  Also, the monodromy $\gs_f = 
\gs_g * \gs_h$.  By the inductive assumption $Y$ is homotopy equivalent 
to $F_1 * F_2 * \cdots * F_{r-1}$, with monodromy $\gs_g = \gs_1 * 
\cdots * \gs_{r-1}$. so by Oka\rq s theorem, $X$ is homotopy equivalent 
to $(F_1 * F_2 * \cdots * F_{r-1}) * F_r$, with monodromy $\gs_f = \gs_1 
* \cdots * \gs_{r-1} * \gs_{r}$.  Hence, by repeated application of 
(\ref{eqn6.3}) we see that $\tilde{H}^*(F; \bk)$ has the indicated form as a 
graded vector space.  \par
Furthermore, for iv) if each global Milnor fibration $p_i : E_i \to S^1$ is 
cohomologically trivial, then $\gs_i \equiv id$ so by iii) $\gs_f \equiv id$ 
and the Milnor fiber of $f$ is cohomologically trivial.  Thus, by Proposition 
\ref{Prop1.5}  $H^{*}(E; \bk)$ is given by (\ref{eqn6.5}).  As $f$ is 
weighted homogeneous, by an analogous argument as in Lemma 
\ref{Lem1.1} it follows that $\C^n \backslash \cE$ has $E$ as a 
deformation retract; thus, the remainder of (\ref{eqn6.5}) follows.  
\end{proof}
\par
We next separately give the form of the (co)homology of the link $L(\cE)$. 
To do so we will introduce some notation. Given a (finite dimensional) 
graded vector space $W = \oplus W_j$ over the field $\bk$, we let 
$\bk[m]$ denote $\bk$ with a grading of degree $m$; and apply a grading to 
$\hom(W, \bk[m]) \simeq \oplus \hom(W_j, \bk[m])$ so that if $W_j$ has 
graded degree $\ell_j$, then $\hom(W_j, \bk[m])$ has graded degree $m - 
\ell_j$.  
\begin{Corollary}
\label{Cor6.5}
If $f$ is given as a formal sum as above and the Milnor fiber of each $f_i$ 
is cohomologically trivial, then 
\begin{multline}
\label{eqn6.6}
\widetilde{H}^{*}(L(\cE); \bk) \,\, \simeq \,\,  \bk[2n-3] \oplus 
\hom(\otimes_{i = 1}^{r} \,\, \tilde{H}^*(M_i; \bk), \bk[2n-2]) [-(r-1)] \\ 
 \, \oplus \, \hom(\otimes_{i = 1}^{r} \,\, \tilde{H}^*(M_i; \bk), \bk[2n-2]) 
[-r] \, .
\end{multline}
\end{Corollary}
\begin{proof}
By v) of Theorem \ref{Thm6.4},
$H^{*}(\C^n \backslash \cE; \bk)$ is given by (\ref{eqn6.5}).  Then, by the 
same argument given in Proposition  \ref{Prop1.8}, $S^{2n-1} \backslash 
L(\cE)$ is homotopy equivalent to $\C^n \backslash \cE$.  Hence, we may 
again apply Alexander duality as (\ref{Eqn1.10}) in the proof of 
Proposition \ref{Prop1.8} to obtain the isomorphism of graded vector 
spaces
\begin{equation}
\label{eqn6.7}
\widetilde{H}^{2n -2-j}(L(\cE); \bk) \,\, \simeq 
\hom(\widetilde{H}_j(S^{2n-1} \backslash L(\cE); \bk), \bk[2n-2])  \, .
\end{equation}
Summing (\ref{eqn6.7}) over $j$ yields an  isomorphism of graded vector 
spaces
\begin{equation}
\label{eqn6.8}
\widetilde{H}^*(L(\cE); \bk) \,\, \simeq \hom(\widetilde{H}_*(S^{2n-1} 
\backslash L(\cE); \bk), \bk[2n-2])
\end{equation}
By (\ref{eqn6.5}), we may decompose the RHS of (\ref{eqn6.8}) as the 
direct sum of graded vector spaces
\begin{multline}
\label{eqn6.10}
   \hom(\bk\{s_1\} ,\bk[2n-2]) \oplus \hom\left( (\otimes_{i = 1}^{r} \, 
\tilde{H}^*(M_i; \bk))[(r-1)], \bk[2n-2]\right)  \\ 
 \, \oplus \, \hom\left( \bk\{s_1\}\otimes (\otimes_{i = 1}^{r} \,\, 
\tilde{H}^*(M_i; \bk))[r-1], \bk[2n-2]\right) \, .
\end{multline}
As graded vector spaces, the first summand is $\bk[2n-3]$, and 
$$  \bk\{s_1\}\otimes (\otimes_{i = 1}^{r}  \,\, \tilde{H}^*(M_i; \bk))[r-
1]\,\, \simeq  \,\, (\otimes_{i = 1}^{r}  \,\, \tilde{H}^*(M_i; \bk))[r] \, .$$
Then, for the second and third summands in (\ref{eqn6.10}), the positive 
shift inside \lq\lq$\hom$\rq\rq\,  can be moved to a negative shift 
outside.  Hence, when we add $\bk\{1\}$ to both sides of (\ref{eqn6.8}), 
we obtain from (\ref{eqn6.7}) and (\ref{eqn6.8}) the isomorphism of 
graded vector spaces (\ref{eqn6.6}).  
\end{proof}
Unfortunately $H^{*}(L(\cE); \bk)$ in this more general case cannot be 
given in such a simple form as a shifted and upper truncated cohomology 
of a compact orientable manifold or an exterior algebra as in iii) of 
Proposition \ref{Prop1.8} and Theorems \ref{Thm2.2} and \ref{Thm3.6}.  
However, we can still view the cohomology as a direct sum of the shifted 
upper truncated exterior algebras which appear in the theorems.  \par
\begin{Example}
As a simple example we consider the formal sum
$$  f \,\, = \,\, a_1 \det \begin{pmatrix} x_1 & x_2 & x_3 \\ x_4 & x_5 & 
x_6 \\  x_7 & x_8 & x_9\end{pmatrix}  + a_2 y_1 y_2  
\det\begin{pmatrix} y_1 & y_2 \\ y_3 & y_4 \end{pmatrix}  $$
which defines a hypersurface singularity in $\C^{13}$.  We denote the first 
term by $f_1(x)$ and the second by $f_2(y)$.  Then, by Theorem 
\ref{Thm2.1}, the Milnor fiber of $f_1$ is homotopy equivalent to  $SU_3$ 
and has cohomology isomorphic to an exterior algebra $\gL^*\bk(e_3, 
e_5)$.  Also, by \cite{DP2} (see also Example \ref{Ex3.9}), $f_2$ defines a 
linear free divisor resulting from the action of a solvable linear algebraic 
group and by \cite{DP} it has Milnor fiber homotopy equivalent to torus, 
with cohomology an exterior algebra $\gL^*\bk(e_1, e_1^{\prime})$.  Thus, 
by Theorem \ref{Thm6.4} the Milnor fiber of $f$ is homotopy equivalent to 
$SU_3 * T^2$ with reduced cohomology 
$\left( \widetilde{\gL}^*\bk(e_3, e_5) \otimes \widetilde{\gL}^*\bk(e_1, 
e_1^{\prime})\right)[1]$, which is $0$ in degrees $< 5$ and in degrees $ 5 
\leq \ell \leq 11$ has dimensions $2, 1, 2, 1, 0, 2, 1$.  We see from the 
dimensions that it is a direct sum of three shifted copies of 
$\widetilde{\gL}^*\bk(e_1, e_1^{\prime})$ with nonzero dimensions $2, 1$ 
or three shifted copies of $\widetilde{\gL}^*\bk(e_3, e_5)$ which has its 
dimensions $1, 0, 1, 0, 0, 1$ between degrees $3$ and $8$.
\par
Also, both Milnor fibrations are cohomologically trivial by Theorems 
\ref{Thm2.1} and \ref{Thm4.2}; thus, by iii) of Theorem \ref{Thm6.4}, the 
Milnor fibration of $f$ is cohomologically trivial.  Hence, by Corollary  
\ref{Cor6.5}, the link $L(\cE)$ has reduced cohomology which is $0$ in 
degrees $< 12$ and in degrees $12 \leq \ell \leq 23$ has dimensions $1, 3, 
2, 1, 3, 3, 3, 2, 0, 0, 0, 1$.  Here we see the group of lower nonzero 
dimensions obtained from those for the Milnor fiber written in reverse 
order $1, 2, 0, 1, 2, 1, 2$ and added to another copy shifted by $1$. 
\end{Example}

\subsection*{Milnor Fibers Homotopy Equivalent to a Bouquet of 
Suspensions of Joins of Compact Manifolds}  \hfill
\par
As a last step, we consider a formal sum $h = f \oplus g$, where $f$ is a 
formal sum of $f_i$ which define exceptional orbit hypersurfaces as 
considered earlier, and $g$ is weighted homogeneous and has a Milnor 
fiber which is homotopy equivalent to a bouquet of spheres $\vee_{i = 
1}^{k} S^{n_i}$.  Then, we may again apply Oka\rq s Theorem to conclude
\begin{Proposition}
\label{Prop6.11}
The Milnor fiber of $h = f \oplus g$ as above has the homotopy type of a 
bouquet of spaces, each of which is an iterated $n_i +1$ suspension 
$S^{n_i+1}(*_{j = 1}^{r} M_j)$ for $i = 1, \dots, k$.
\end{Proposition}
\begin{Remark}
\label{Rem6.12}
Thus, within the class of nonisolated hypersurface singularities obtained 
from formal sums, the Milnor fibers which are a bouquet of spheres are 
replaced more generally by bouquets of spaces each of which are 
suspensions of joins of compact manifolds.  
\end{Remark}
\begin{proof}
Let $X = *_{j = 1}^{r} M_j$,which is homotopy equivalent to the Milnor 
fiber of $f$.  By Oka\rq s Theorem, the Milnor fiber of $h$ is homotopy 
equivalent to the join $X*(\vee_{i = 1}^{k} S^{n_i})$.  Let $p$ denote the 
common point of the bouquet.  Then, this join is the union of the joins $X * 
S^{n_i}$, and any two intersect in the common subspace $X * \{ p\}$.  This 
is a cone on $X$, and is hence contractible.  Moreover, let $U_i \subset  
S^{n_i}$ be a contractible neighborhood of $p$, with $\varphi_i : U_i \to \{ 
p\}$ a strong deformation retraction.  Then, the join of  $\varphi_i$ with 
the identity on $X$ gives a strong deformation retraction $\tilde 
\varphi_i$ of $X * U_i$ to $X * \{ p\}$.  Thus, together these give a strong 
deformation retraction of $\cup_{i = 1}^{k} X * U_i$ to $X * \{ p\}$.  As 
$X * \{ p\}$ has $\{ p\}$ as a strong deformation retract, we conclude that 
$\cup_{i = 1}^{k} X * U_i$ has $\{p\}$ as a strong deformation retract.  
Hence, we may collapse $\cup_{i = 1}^{k} X * U_i$ to $\{p\}$ and retain the 
same homotopy type.  It follows that $X * (\vee_{i = 1}^{k} S^{n_i})$ is 
homotopy equivalent to $\vee_{i = 1}^{k} (X *S^{n_i})$, and each $X 
*S^{n_i}$ is homeomorphic to the $n_i +1$ iterated suspension 
$S^{n_i+1}(X)$.  Together these give the result.  
\end{proof}

\end{document}